\pdfoutput=1
\documentclass[a4paper,reqno]{amsart}
\usepackage{mathptmx}
\usepackage[utf8]{inputenc}
\usepackage{color}
\usepackage{prettyref}
\usepackage{float}
\usepackage{mathtools}
\usepackage{amsbsy}
\usepackage{amstext}
\usepackage{amsthm}
\usepackage{amssymb}
\usepackage{stmaryrd}
\usepackage{microtype}
\usepackage[bookmarks=false,
 breaklinks=false,pdfborder={0 0 1},backref=false,colorlinks=true]
 {hyperref}
\hypersetup{
 pdfstartview=FitV,linkcolor=black,citecolor=black}

\makeatletter

\pdfpageheight\paperheight
\pdfpagewidth\paperwidth

\newcommand{\noun}[1]{\textsc{#1}}

\numberwithin{equation}{section}
\numberwithin{figure}{section}
\theoremstyle{plain}
\newtheorem{thm}{\protect\theoremname}[section]
\theoremstyle{plain}
\newtheorem{prop}[thm]{\protect\propositionname}
\theoremstyle{remark}
\newtheorem{rem}[thm]{\protect\remarkname}
\newenvironment{lyxlist}[1]
	{\begin{list}{}
		{\settowidth{\labelwidth}{#1}
		 \setlength{\leftmargin}{\labelwidth}
		 \addtolength{\leftmargin}{\labelsep}
		 }}
	{\end{list}}
\theoremstyle{remark}
\newtheorem*{rem*}{\protect\remarkname}
\theoremstyle{plain}
\newtheorem{lem}[thm]{\protect\lemmaname}
\theoremstyle{plain}
\newtheorem{cor}[thm]{\protect\corollaryname}

\usepackage{bbm}
\usepackage{pgfplots}
\pgfplotsset{compat=1.15}
\usepackage{mathrsfs}
\usetikzlibrary{arrows}
\usepackage{yfonts}
\usepackage{txfonts}
\usepackage{textcomp}
\usepackage{xcolor}
\usepackage{dsfont}
\usepackage{graphicx}
\usepackage{caption}
\usepackage{subcaption}
\usepackage{float}
\usepackage{tikz}
\usepackage{tikz-cd}
\usepackage[foot]{amsaddr}
\usepackage{enumitem}

\renewcommand{\d}{\;\mathrm{d}}
\renewcommand{\rho}{\varrho}

\newcommand{\1}{\mathbbm{1}}

\newcommand{\e}{\mathrm{e}}

\DeclareMathOperator{\card}{card}

\renewcommand{\emptyset}{\varnothing}
\renewcommand{\phi}{\varphi}
\newcommand{\N}{{\mathbb{N}}}
\def\R{\mathbb{R}}
\def\Z{\mathbb{Z}}
\def\H{\mathbb{H}}

 \usepackage{todonotes}

\AtBeginDocument{
\DeclareFieldFormat[article]{title}{{#1}}
\DeclareFieldFormat[unpublished]{title}{{#1}}
\DeclareFieldFormat[proceedings]{title}{{#1}}
\DeclareFieldFormat[incollection]{title}{{#1}}
\DeclareFieldFormat[unpublished]{title}{{\em #1}}
\renewbibmacro{in:}{}
\DeclareFieldFormat[article]{pages}{#1}
}

\DeclareMathOperator{\cov}{cov}

\renewcommand{\emptyset}{\varnothing}

\newcommand{\Rep}{\boldsymbol{\frak{R}}}
\newcommand{\Per}{{\mathfrak{P\!e\!r}}}
\newcommand{\RepMU}{\boldsymbol{\mathfrak{R}}'}
\newcommand{\FI}{F} 

\theoremstyle{definition}
\newtheorem*{assump*}{Assumption}
\newrefformat{sec}{Section~\ref{#1}}
\newrefformat{subsec}{Section~\ref{#1}}
\newrefformat{prop}{Proposition~\ref{#1}}
\newrefformat{cor}{Corollary~\ref{#1}}
\newrefformat{exa}{Example~\ref{#1}}
\newrefformat{thm}{Theorem~\ref{#1}}
\newrefformat{fig}{Figure~\ref{#1}}

\definecolor{lime}{HTML}{A6CE39}
\DeclareRobustCommand{\orcidicon}{%
	\begin{tikzpicture}
	\draw[lime, fill=lime] (0,0) 
	circle [radius=0.16] 
	node[white] {{\fontfamily{qag}\selectfont \tiny ID}};
	\draw[white, fill=white] (-0.0625,0.095) 
	circle [radius=0.007];
	\end{tikzpicture}
	\hspace{-2mm}
}

\foreach \x in {A, ..., Z}{%
	\expandafter\xdef\csname orcid\x\endcsname{\noexpand\href{https://orcid.org/\csname orcidauthor\x\endcsname}{\noexpand\orcidicon}}
}

\makeatother

\usepackage[style=authoryear,backend=bibtex,style=alphabetic,  backref=false, firstinits=true,citestyle=alphabetic, natbib=false, url=false, isbn=false,doi=true,eprint=false,maxbibnames=99]{biblatex}
\providecommand{\corollaryname}{Corollary}
\providecommand{\lemmaname}{Lemma}
\providecommand{\propositionname}{Proposition}
\providecommand{\remarkname}{Remark}
\providecommand{\theoremname}{Theorem}

\begin{document}
\title[Dimension gap and phase transition for reflected random walks]{ Dimension gap and phase transition for one-dimensional random walks
with reflective boundary}
\author{Maik Gröger\orcidA{}}
\address{Faculty of Mathematics and Computer Science, Jagiellonian University,
ul. Prof. S. Łojasiewicza 6, 30-348 Kraków, Poland}
\email{maik.groeger@im.uj.edu.pl}
\author{Johannes Jaerisch\orcidB{}}
\address{Graduate School of Mathematics, Nagoya University, Furocho, Chikusaku,
Nagoya, 464-8602 Japan}
\email{jaerisch@math.nagoya-u.ac.jp}
\author{Marc Kesseböhmer\orcidC{}}
\address{Institute for Dynamical Systems, Faculty 3 -- Mathematics and Computer
Science, University of Bremen, Bibliothekstr. 5, 28359 Bremen, Germany}
\email{mhk@uni-bremen.de}
\begin{abstract}
We study $\Z$- and $\N$-extensions of interval maps with at most
countably many full branches modelling one-dimensional random walks
without and with a reflective boundary. We analyse the associated
Gurevich pressure and explore the relations governing these two cases.
For such extensions, we obtain variational formul\ae~for the Gurevich
pressure that depend only on the base system. As a consequence, we
characterise the systems with a dimension gap and, in the presence
of a reflective boundary, provide general conditions in terms of asymptotic
covariances for a second order phase transition. As a by-product,
we derive a variational formula for the spectral radius of infinite
Hessenberg matrices.
\end{abstract}

\keywords{Transient dynamics, skew-products, thermodynamic formalism, Poincaré
exponent\emph{,} phase transition, random walk, reflective boundary,
Hessenberg matrix.}
\subjclass[2000]{37E05, 37D35, 60J10, 28A80}
\date{\today}
\maketitle

\section{Introduction}

Interval maps, including unimodal maps such as the logistic map, are
fundamental to the study of chaotic behaviour and bifurcation theory.
In essence, interval maps represent a cornerstone in the field of
dynamical systems, offering a unifying framework for grasping a vast
spectrum of dynamical phenomena. In particular, they can be employed
to model scenarios where the subsequent state--say, in the integers--hinges
on a random selection drawn from an interval base transformation on
the unit interval. Such models can be used to elucidate the influence
of randomness on long-term behaviour and on limiting distributions.

The present study focuses on the dimension theory and thermodynamic
formalism for interval maps that are not covered by the symbolic thermodynamic
formalism of mixing Markov shifts with finite alphabet, nor by its
generalisations to countable Markov shifts with finitely irreducible
incidence matrix or the closely related concept of big images and
preimages \cites{MR2003772}{MR1738951}. In \cite{MR4373996} the
authors studied the transient and recurrent dynamics of skew-product
dynamical systems that are transitive but not mixing and whose base
transformations are given by an interval map with \emph{finitely}
many expanding full branches, modelling in this way a \emph{random
walk on $\R$}. The purpose of this paper is to generalise this setting
in two ways: First, for the base dynamics, we allow infinitely many
branches; second, and more importantly, we model a \emph{random walk
with a boundary}, that is, \emph{on} $\R_{\geq0}$.

In addition to this fundamental treatment of random walks with a boundary,
our approach also allows us to embed interesting examples--such as
those discussed in \cite{MR2959300,MR1438267,MR3289914}, namely dissipative
interval dynamics and (linear) models of induced maps arising from
unimodal maps--into this more systematic framework. 

The primary objective of this study is to address the phenomenon of
a dimension gap (see \prettyref{subsec:Expanding-interval-maps} for
the definition). This is a crucial aspect that necessitates a detailed
examination of semi-group extensions, such as $\Z$- and $\N$-extensions,
see \prettyref{thm:DimGap_non-reflective}, \prettyref{thm:HD_Transient}
and \prettyref{thm:Big_Theorem}. Fixing an invariant probability
measure on the unit interval and considering the stochastic process
obtained by projecting the dynamics on the real line to the integer
part, our $\mathbb{N}$-extensions can be regarded as random walks
on $\mathbb{Z}$ with a reflective boundary in zero. This in turn
gives rise to the associated phenomenon of a second-order phase transition,
which we will characterise in comprehensive detail in \prettyref{thm:Phasetransition}.

Our motivation stems, in part, from a specific family of piecewise
linear maps suggested by van Strien, explored by Stratmann and Vogt
\cite{MR1438267}, and later by Bruin and Todd \cite{MR2959300}.
This family provides a piecewise linear modelling of induced maps
arising from unimodal maps and captures their transient and recurrent
behaviour (see also \cite{MR3289914}). We will revisit this family
in \prettyref{sec:Applications-and-examples}, providing alternative
proofs of previously established properties of these maps and opening
a new perspective on dimension gaps, strange attractors, and second-order
phase transitions in the context of these piecewise linear models.
Although the examples discussed in \prettyref{sec:Applications-and-examples}
are mainly of linear type, our results apply more generally to non-linear
induced maps as well.

Dimension gaps arise in several other prominent settings, including
those studied by Avila and Lyubich \cite{MR2373353}, Mayer and Urbański
\cite{MR4709229}, and Ledrappier and Lessa \cite{LedrappierLessa2025}.
In the setting of our models of random walks with a reflective boundary,
we can completely describe the emergence of a dimension gap and its
quantitative relation to the unreflected case. Furthermore, the appearance
of a second-order phase transition--first observed in \cite{MR2959300}
for the above-mentioned family--and its interpretation as a boundary
effect, together with its characterisation via asymptotic covariance
criteria, represent a novel contribution.

Our work offers a general perspective that unifies the concept of
Gurevich pressure across the two types of extensions under consideration.
We demonstrate how this pressure can be effectively calculated through
a variational approach (see \prettyref{eq:(2)} and \prettyref{eq:(3)}),
relying solely on the pressure function for the base system. As a
by-product, this also yields a variational formula for the spectral
radius of infinite Hessenberg matrices (see the end of \prettyref{sec:Applications-and-examples}).
Moreover, our approach is well in line with studies on random walks
in random environments (see, for example, \cites{MR657919}{MR678156}{MR4674045}{MR3816760})
and may provide a useful framework for exploring related phenomena
in that broader setting.

\section{Exposition of main results}

We first introduce a specific class of expanding interval maps on
$\R$ and state a version of a well-known result of Bishop and Jones
for Kleinian groups.

\subsection{Expanding interval maps and recurrent sets\label{subsec:Expanding-interval-maps}}

We say that $f:\Delta\rightarrow\R$ is an \emph{expanding $C^{2}$-Markov
interval map} if it satisfies the following conditions:
\begin{enumerate}
\item (\emph{Open domain}) There exist $I\subset\N$ with $\card(I)\ge2$
and non-empty pairwise disjoint open intervals $(J_{i})_{i\in I}$
such that $\Delta=\biguplus_{i\in I}J_{i}$.
\item (\emph{Grid structure}) We have $\Delta\cap\Z=\emptyset$ and for
each $i\in I$ there exists $k\in\Z$ such that $f(J_{i})=(k,k+1)$.
\item (\emph{$C^{2}$-extension}) For each $i\in I$, $f|_{J_{i}}$ has
a $C^{2}$-extension to a neighbourhood of the closure $\overline{J_{i}}$,
whose restriction to $\overline{J_{i}}$ is denoted by $f_{i}$.
\item (\emph{Uniform expansion}) It holds that $\inf_{x\in\Delta}\left|f'\left(x\right)\right|>1$.
\item (\emph{Rényi condition}) We have $\sup_{x\in\Delta}\left|f''(x)\right|\big/\left(f'(x)\right)^{2}<\infty$.
\end{enumerate}
We remark that the $C^{2}$-extension and Rényi condition assumptions
can be relaxed to a $C^{1+\varepsilon}$-extension combined with a
bounded distortion condition on $f'$.

The\emph{ repeller} of $f:\Delta\rightarrow\R$ is defined by
\[
\Rep\left(f\right)\coloneqq\bigcap_{n\ge0}f^{-n}\left(\Delta\right).
\]
Comparing this notion with the \emph{limit set }of the corresponding\emph{
graph directed Markov system} in the terminology of \cite{MR2003772}
with possibly infinitely many vertices, 
\[
\RepMU(f)\coloneqq\adjustlimits\bigcap_{n\ge0}\bigcup_{\left(\omega_{0},\dots,\omega_{n}\right)\in I^{n+1}:f\left(J_{\omega_{i}}\right)\supset J_{\omega_{i+1}}}f_{\omega_{0}}^{-1}\circ\dots\circ f_{\omega_{n}}^{-1}\left(f_{\omega_{n}}\left(\overline{J_{\omega_{n}}}\right)\right),
\]
 we find that $\RepMU(f)\setminus\Rep(f)$ is at most countable. Since
\[
f\left(\Rep\left(f\right)\right)=f\left(\bigcap_{n\ge0}f^{-n}(\Delta)\right)\subset f\left(\Delta\right)\cap\bigcap_{n\ge0}f^{-n}\left(\Delta\right)\subset\Rep\left(f\right),
\]
the repeller gives rise to the dynamical system $f:\Rep\left(f\right)\to\Rep\left(f\right)$,
which, throughout, we assume to be topologically transitive:
\begin{enumerate}
\item [(6)] (\emph{Topological transitivity}) For all $U,V$ non-empty
(relative) open subsets in $\Rep\left(f\right)$ there exists $n\in\N\coloneqq\left\{ 0,1,2,\ldots\right\} $
such that $f^{n}(U)\cap V\neq\emptyset$.
\end{enumerate}
If $\Rep\left(f\right)$ is unbounded, then the natural question arises
about recurrent and transient behaviour. For this we define the \emph{recurrent
}and \emph{uniformly recurrent sets }
\[
R\left(f\right)\coloneqq\left\{ x\in\Rep\left(f\right):\liminf_{n\to\infty}\left|f^{n}(x)\right|<\infty\right\} ,\;R_{u}\left(f\right)\coloneqq\left\{ x\in\Rep\left(f\right):\limsup_{n\to\infty}\left|f^{n}(x)\right|<\infty\right\} .
\]
For $n\geq1,$ let
\[
\Per_{n}\left(f\right)\coloneqq\left\{ x\in\Rep\left(f\right):f^{n}(x)=x\right\} 
\]
 denote the set of all $n$-periodic points of $f:\Rep\left(f\right)\to\Rep\left(f\right)$.
For fixed $i\in I$ we define the \emph{critical (or Poincaré) exponent}
for the recurrent dynamics by
\[
\delta\left(f\right)\coloneqq\inf\bigg\{ s\ge0:\sum_{n\ge1}\sum_{x\in\Per_{n}\left(f\right)\cap J_{i}}\left|\left(f^{n}\right)'(x)\right|^{-s}<\infty\bigg\}.
\]
It follows from a version of \eqref{eq:delta_vs_pressure} combined
with \prettyref{prop:gurevich_wrt_f} that $\delta\left(f\right)$
can be expressed in terms of the Gurevich pressure for countable Markov
shifts. Therefore, we deduce that our definition of $\delta\left(f\right)$
is indeed independent of the particular choice of $i\in I$.

The following dimension formula for the (uniformly) recurrent set
is an analog of a well-known theorem of Bishop and Jones for Kleinian
groups \cite{MR1484767}, making use of the Poincaré exponent. The
method of proof is by now well established also in the framework of
countable Markov shifts and conformal iterated function systems \cite{MR4553466,MR3610938,MR2003772}.
The upper bound is a simple covering argument based on the definition
of the critical exponent. The lower bound uses the exhaustion principle
for the Gurevich pressure (\prettyref{lem:exhaustio_principle}) together
with Bowen's formula for conformal iterated function systems \cite[Theorem 4.2.13]{MR2003772}.
\begin{prop}
\label{prop:Bowen_for-the-skew-periodic-1} For an expanding $C^{2}$-Markov
interval map $f$ we have
\[
\dim_{H}\left(R_{u}\left(f\right)\right)=\dim_{H}\left(R\left(f\right)\right)=\delta(f),
\]
where $\dim_{H}\left(A\right)$ refers to the Hausdorff dimension
of $A\subset\mathbb{R}$.
\end{prop}

We say that an expanding $C^{2}$-Markov interval map $f$ exhibits
a \emph{dimension gap} if 
\[
\dim_{H}\left(R\left(f\right)\right)<\dim_{H}\left(\Rep\left(f\right)\right).
\]
This phenomenon can also be expressed in terms of the \emph{transient
set} 
\[
T\left(f\right)\coloneqq\left\{ x\in\Rep\left(f\right):\liminf_{n\to\infty}\left|f^{n}(x)\right|=+\infty\right\} .
\]
Obviously, $\Rep\left(f\right)$ equals the disjoint union $R\left(f\right)\uplus T\left(f\right)$
and hence 
\[
\dim_{H}\left(\Rep\left(f\right)\right)=\max\left\{ \dim_{H}R\left(f\right),\dim_{H}T\left(f\right)\right\} ,
\]
which implies that a dimension gap occurs if and only if 
\[
\dim_{H}\left(R\left(f\right)\right)<\dim_{H}\left(T\left(f\right)\right).
\]

The occurrence of a dimension gap can finally be stated in the spirit
of the \emph{hyperbolic dimension} \cite{MR1626737,MR4709229}: $f$
has a dimension gap if and only if 
\[
\sup\left\{ \dim_{H}\left(\Rep\left(f|_{\bigcup_{i\in K}J_{i}}\right)\right):K\subset I,2\le\card\left(K\right)<\infty\right\} <\dim_{H}\left(\Rep\left(f\right)\right).
\]
Well-known sufficient conditions for the absence of a dimension gap
include finite irreducibility (as defined for the symbolic coding
of the Markov map $f$) or the big images and preimages property.
Note however that these conditions are satisfied if and only if $\Rep(f)$
is bounded. In this paper, we will only focus on systems where $\Rep(f)$
is unbounded and therefore finite irreducibility or the big images
and preimages property cannot hold.

The phenomenon of a dimension gap in the context of dynamical systems
is referred to in various places in the literature and is sometimes
also called a \emph{dimension drop}. We would like to mention \cite{MR2373353},
\cite{MR4709229} and \cite{MR4587901,LedrappierLessa2025,MR4609146}
for some recent contributions.

\subsection{$\Psi$-extensions and skew-periodic interval maps \label{subsec:Periodic-setting-and}}

From now on let $\FI_{0}:\Delta_{0}\to\left(0,1\right)$ be an expanding
$C^{2}$-Markov interval map on the unit interval, meaning $\Delta_{0}\coloneqq\biguplus_{i\in I}J_{i}\subset(0,1)$.
For $\H\subset\Z$ we define the \emph{Minkowski sum} $\Delta_{\H}\coloneqq\Delta_{0}+\H$
with the corresponding index set $I_{\H}\coloneqq I\times\H$ and
call $\Psi:\Delta_{\H}\rightarrow\H$ a \emph{step function} if $\Psi|_{J_{i}+h}$
is constant for each $\left(i,h\right)\in I_{\H}$. We define the
\emph{$\Psi$-extension} of $\FI_{0}$ by 
\[
\FI_{\Psi}:\Delta_{\H}\rightarrow\R,\quad x\mapsto\FI_{0}\left(x-\left\lfloor x\right\rfloor \right)+\Psi(x),
\]
 where $\left\lfloor x\right\rfloor \coloneqq\max\left\{ n\in\Z:n\leq x\right\} $.
In this way we obtain an expanding $C^{2}$-Markov interval map with
\[
\Rep(\FI_{\Psi})=\Rep(\FI_{0})+\H\eqqcolon\Rep_{\H}
\]
and as before we assume that $\FI_{\Psi}:\Rep(\FI_{\Psi})\to\Rep(\FI_{\Psi})$
is topologically transitive. We use $0$ to denote the one-element
subgroup of $\Z$ and with this notation we have $\Rep_{0}=\Rep\left(\FI_{0}\right)=R\left(\FI_{0}\right)$.
Further, with $\kappa\left(x\right)\coloneqq x-\left\lfloor x\right\rfloor $,
$x\in\R$ denoting the projection to $\left[0,1\right)$, the skew-product
structure can be read from the following commutative diagram:
\begin{center}
\vspace*{-1ex}
\begin{tikzcd}
  \Rep_{\H} \arrow[r, "\FI_{\Psi}"] \arrow[d, "\kappa"']
    & \Rep_{\H} \arrow[d, "\kappa" ] \\
  \Rep_{0} \arrow[r,  "\FI_0"'  ]
& \Rep_{0} \end{tikzcd}
\par\end{center}

Now, we will introduce our two main special cases for $\H=\N$ or
$\Z$, which we will force to fulfil some periodicity condition. Fix
the \emph{step length function} $\psi:\Delta_{0}\to\Z_{\geq-1}$,
which is assumed to be constant on $J_{i}$ for $i\in I$.
\begin{description}
\item [{Case~$\boldsymbol{\H=\Z}$}] The following maps where studied
in \cite{MR4373996}: Let
\begin{align*}
\Psi_{\Z}:\Delta_{\Z}\to\Z,\quad & x\mapsto\psi\left(x-\left\lfloor x\right\rfloor \right)+\left\lfloor x\right\rfloor 
\end{align*}
and call $\FI_{\Z}:=\FI_{\,\Psi_{\Z}}$ a \emph{skew-periodic interval
map} \emph{without a reflective} \emph{boundary}.
\item [{Case~$\boldsymbol{\H=\N}$}] We fix additionally the \emph{reflexion
rule} $\psi_{1}:\Delta_{0}\rightarrow\N$, also constant on each $J_{i}$,
$i\in I$, and such that $M+\psi\geq\psi_{1}\ge\psi$ for some $M>0$.
Then, we define
\begin{align*}
\Psi_{\N}:\Delta_{\N}\to\N,\quad & x\mapsto\psi_{1}\left(x\right)\1_{\left(0,1\right]}\left(x\right)+\left(\psi\left(x-\left\lfloor x\right\rfloor \right)+\left\lfloor x\right\rfloor \right)\1_{(1,\infty)}
\end{align*}
and call $\FI_{\N}:=\FI_{\Psi_{\N}}$ a \emph{skew-periodic interval
map} \emph{with a (left) reflective boundary}. Our fractal-geometric
results are independent of the particular choice of $\psi_{1}$, as
we will see later.
\end{description}
\begin{rem}
The notion of a left reflective boundary easily carries over to the
case of a\emph{ right reflective boundary}. Generalisations to step
length functions satisfying $\psi\ge-m$, for some $m\in\mathbb{N}$,
are also possible. For related results for the asymptotics of transition
probabilities for a reflected random walk on $\N$ we refer to \cite{MR1340828}.
\end{rem}

For an expanding $C^{2}$-Markov interval map $f$ let us define the
\emph{right }and\emph{ left transient sets} 
\[
T^{+}\left(f\right)\coloneqq\left\{ x\in\Rep\left(f\right):\liminf_{n\to\infty}f^{n}(x)=+\infty\right\} ,\quad T^{-}\left(f\right)\coloneqq\left\{ x\in\Rep\left(f\right):\limsup_{n\to\infty}f^{n}(x)=-\infty\right\} .
\]
A comparison with the definition of $T$ gives us
\[
T\left(F_{\Z}\right)=T^{+}\left(F_{\Z}\right)\,\uplus\,T^{-}\left(F_{\Z}\right)\qquad\text{and }\qquad T\left(F_{\N}\right)=T^{+}\left(F_{\N}\right).
\]

 For ease of notation, for $\H\in\left\{ 0,\N,\Z\right\} $, we
write
\[
\delta_{\H}:=\delta\left(\FI_{\H}\right).
\]

Henceforth, it will be assumed that $\FI_{0}$ is \emph{strongly regular}
(see \prettyref{subsec:Gibbs-measures-and}). This technical assumption
will be used in the proofs of our main results and implies in particular
that the Hausdorff dimension of the unique equilibrium state $\mu_{\delta_{0}\varphi}$
for $\delta_{0}\varphi$ with \emph{geometric potential} $\varphi:x\mapsto-\log\left|\FI_{0}'(x)\right|$
on $\Rep_{0}$ is equal to $\delta_{0}$. Note that by Proposition
\ref{prop:Bowen_for-the-skew-periodic-1}, $\delta_{0}$ equals the
Hausdorff dimension of $\Rep_{0}$. We will always assume that 
\[
\int\varphi\d\mu_{\delta_{0}\varphi}<+\infty.
\]
To exclude trivial cases we also assume that $-1=\underline{\psi}\coloneqq\inf\psi<0<\overline{\psi}\coloneqq\sup\psi$
and define 
\[
\alpha_{\max}\coloneqq\int\psi\d\mu_{\delta_{0}\varphi}\in\R\cup\left\{ +\infty\right\} .
\]

\begin{rem}
\label{rem:closure_of_repeller} Let us briefly discuss the closure
of the repellers. Denote by $\Rep(\FI_{0},\infty)$ the set of accumulation
points of sequences $(x_{n})$ in $[0,1]$ such that $x_{n}\in\overline{J_{i_{n}}}$
for pairwise distinct $i_{n}\in I$. Further, let us write $\RepMU_{\H}\coloneqq\RepMU(\FI_{\H})$.
It then follows from \cite[Lemma 1.0.1]{MR2003772} that 
\[
\overline{\Rep_{0}}=\overline{\RepMU_{0}}=\RepMU_{0}\cup\bigcup_{n\in\N}\FI_{0}^{-n}\left(\Rep(\FI_{0},\infty)\right).
\]
Moreover, we have for $\H\in\left\{ 0,\N,\Z\right\} $, 
\[
\overline{\Rep_{\H}}=\overline{\RepMU_{\H}}=\overline{\RepMU_{0}}+\H.
\]
If $I$ is finite, then both $\RepMU_{0}$ and $\RepMU_{\H}$ are
closed. If $\FI_{0}$ is, for example, given by the infinitely branched
Gauss map or the Lüroth maps, then $\Rep(\FI_{0},\infty)$ is a singleton
and thus, $\overline{\RepMU_{0}}\setminus\RepMU_{0}$ is a countable
set. In this case, it follows that 
\[
\overline{\RepMU_{\H}}\setminus\RepMU_{\H}=\left(\overline{\RepMU_{0}}\setminus\RepMU_{0}\right)+\H
\]
 is also countable. In particular, in all these cases, 
\[
\dim_{H}\big(\overline{\Rep_{\H}}\big)=\dim_{H}\left(\Rep_{\H}\right)=\dim_{H}\left(\Rep_{0}\right).
\]
\end{rem}

\subsection{Dimension gaps and transient sets\label{subsec:Dimension-gap}}

For $\H\in\left\{ 0,\N,\Z\right\} $ we have by \prettyref{prop:Bowen_for-the-skew-periodic-1}
that $\delta_{\H}=\dim_{H}\left(R\left(\FI_{\H}\right)\right)$. Hence,
$\FI_{\H}$ exhibits a dimension gap if and only if $\delta_{\H}<\delta_{0}$.

It follows from \cite[Theorems C and D]{MR2373353} that for complex
Feigenbaum maps, a dimension gap between the hyperbolic dimension
of the Julia set and its Hausdorff dimension is equivalent to positive
area of the Julia set. We will see that this observation is--considering
the appropriate inducing scheme--reflected by our trichotomy statement
in \prettyref{thm:Big_Theorem}. For skew-periodic interval maps,
criteria for such dimension gaps for both the reflective and non-reflective
case will be presented next. Our first main result reads as follows.
\begin{thm}
\label{thm:DimGap_non-reflective} A skew-periodic interval map $\FI_{\Z}$
exhibits a dimension gap if and only if, either $\alpha_{\max}>0$,
or 
\[
\ensuremath{\alpha_{\max}<0}\quad\text{\&}\quad\exists q>0:\sum_{x\in\Per_{1}\left(\FI_{0}\right)}\exp\left(\delta_{0}\varphi(x)+q\psi(x)\right)<\infty.
\]
\end{thm}

In \prettyref{subsec:Drift-without-dimension}, we will give an example
for a system that shows no dimension gap despite the presence of negative
drift. The asymmetry in \prettyref{thm:DimGap_non-reflective} is
caused by our assumption that the step length function, while bounded
from below, is not necessarily bounded from above.
\begin{thm}
\label{thm:HD_Transient}For the transient sets of skew-periodic interval
maps $\FI_{\Z}$ and $\FI_{\N}$ we have 
\[
\dim_{H}\left(T^{+}\left(\FI_{\N}\right)\right)=\dim_{H}\left(T^{+}\left(\FI_{\Z}\right)\right)\;\text{ and }\;\dim_{H}\left(T^{-}\left(\FI_{\N}\right)\right)=0.
\]
Moreover,
\[
\dim_{H}\left(T^{\pm}\left(\FI_{\Z}\right)\right)=\begin{cases}
\delta_{0}, & \pm\alpha_{\max}\ge0,\\
\delta_{\Z}, & \pm\alpha_{\max}<0.
\end{cases}
\]
\end{thm}

Our next main result is connecting the reflective with the non-reflective
case for skew-periodic interval maps. The equilibrium measures involved
here are considered as Borel measures on $\R$ with support contained
in the unit interval.
\begin{thm}
\label{thm:Big_Theorem} For the critical exponent of a skew-periodic
interval map $\FI_{\N}$ we have 
\[
\text{\ensuremath{\delta}}_{\N}=\begin{cases}
\delta_{\Z}, & \alpha_{\max}\ge0,\\
\delta_{0}, & \alpha_{\max}<0
\end{cases}
\]
and the following trichotomy holds:
\begin{description}
\item [{\emph{Lean~Case}}] ~~~~~~~~~~~~~$\dim_{H}\left(R\left(\FI_{\N}\right)\right)>\dim_{H}\left(T^{+}\left(\FI_{\N}\right)\right)$
\begin{lyxlist}{00.00.0000}
\item [{\hspace*{2cm}$\iff$$\alpha_{\max}<0$}] $\&$ $\exists q>0:\sum_{x\in\Per_{1}\left(\FI_{0}\right)}\e^{\delta_{0}\varphi(x)+q\psi\left(x\right)}<\infty$
\item [{\hspace*{2cm}$\implies$$\mu_{\delta_{0}\varphi}\left(R\left(\FI_{\N}\right)\right)=1$.}]~
\end{lyxlist}
\item [{\emph{Balanced~Case}}] $\dim_{H}\left(R\left(\FI_{\N}\right)\right)=\dim_{H}\left(T^{+}\left(\FI_{\N}\right)\right)$
\begin{lyxlist}{00.00.0000}
\item [{\hspace*{2cm}$\iff$$\alpha_{\max}=0$,}] or $\alpha_{\max}<0$
$\&$ $\forall q>0:\sum_{x\in\Per_{1}\left(\FI_{0}\right)}\e^{\delta_{0}\varphi(x)+q\psi\left(x\right)}=\infty$
\item [{\hspace*{2cm}$\implies$}] $\mu_{\delta_{0}\varphi}\left(R\left(\FI_{\N}\right)\right)=1$.
\end{lyxlist}
\item [{\emph{Black~Hole~Case}}] $\dim_{H}\left(R\left(\FI_{\N}\right)\right)<\dim_{H}\left(T^{+}\left(\FI_{\N}\right)\right)$
\begin{lyxlist}{00.00.0000}
\item [{\hspace*{2cm}$\iff$$\dim_{H}\left(R\left(\FI_{\N}\right)\right)<\dim_{H}\left(\Rep\left(\FI_{\N}\right)\right)$}]~
\item [{\hspace*{2cm}$\iff$$\alpha_{\max}>0$}]~
\item [{\hspace*{2cm}$\iff$$\mu_{\delta_{0}\varphi}\left(T^{+}\left(\FI_{\N}\right)\right)=1$.}]~
\end{lyxlist}
\end{description}
In particular, $\FI_{\N}$ exhibits a dimension gap if and only if
$\alpha_{\max}>0$, or equivalently, if $\mu_{\delta_{0}\varphi}\left(T^{+}\left(\FI_{\N}\right)\right)=1$.
\end{thm}

We have named the three cases in line with Avila and Lyubich \cite{MR2373353},
who studied the dynamics of complex Feigenbaum maps with periodic
combinatorics. Note that for both the lean and the balanced case we
necessarily have $\dim_{H}\left(R\left(\FI_{\N}\right)\right)=\delta_{0}$,
whereas in the black hole case $\delta_{0}=\dim_{H}\left(T^{+}\left(\FI_{\N}\right)\right)$.

In accordance with contributions in \cite[Theorem E]{MR2373353},
we can distinguish two cases with respect to the conformal measure.
If the $\delta_{0}$ conformal measure (in this context, the Gibbs
measure $\mu_{\delta_{0}\varphi}$) is conservative--that is, $\mu_{\delta_{0}\varphi}\left(R\left(\FI_{\N}\right)\right)=1$--then
the hyperbolic dimension and Hausdorff dimension of the `Julia set'
are equal, i.\,e\@., in our situation, we do not observe a dimension
gap, $\dim_{H}\left(R\left(\FI_{\N}\right)\right)=\dim_{H}\left(\Rep\left(\FI_{\N}\right)\right)$.
Conversely, if $\mu_{\delta_{0}\varphi}$ is dissipative--that is,
$\mu_{\delta_{0}\varphi}\left(T^{+}\left(\FI_{\N}\right)\right)=1$--then
the system exhibits a dimension gap, in our situation, $\dim_{H}\left(\Rep\left(\FI_{\N}\right)\right)>\dim_{H}\left(R\left(\FI_{\N}\right)\right)$.

Our trichotomy is also closely related to the one in \cite[Theorem 7]{MR3289914}.
Note that in our setting, in the balanced case, the drift could be
negative. This possibility is excluded in \cite[Theorem 7]{MR3289914}
by the \emph{`good drift property'}, namely that the probability
of $\left\{ \psi\ge k\right\} $ decays exponentially, whereas in
our example provided in \prettyref{subsec:Drift-without-dimension},
the probability of $\left\{ \psi\ge k\right\} $ decays only polynomially.
Hence, the trichotomy can be expressed under the good drift property
by $\alpha_{\max}<0$, $\alpha_{\max}=0$, and $\alpha_{\max}>0$.

Lastly, our findings align with the corresponding real Feigenbaum
dynamics of unimodal maps. As announced at the very beginning of this
paper, we will compare our results in detail with those in \cite{MR1438267,MR2959300,MR3610938}
below in \prettyref{sec:Applications-and-examples}: For $F_{\N}=L_{\lambda,\N}$
with $\lambda\in\left(0,1\right)$, we show that the three cases from
\prettyref{thm:Big_Theorem} can be distinguished by $\lambda<1/2$,
$\lambda=1/2$, and $\lambda>1/2$. In particular, $\lambda\in\left(1/2,1\right)$
corresponds to the black hole case, which is related to the existence
of a \emph{wild attractor} at $\left\{ +\infty\right\} $. This follows
from transitivity: Generic points have a dense orbit, while for Lebesgue
almost every $x\in\Rep\left(L_{\lambda,\N}\right)$, the $\omega$-limit
set satisfies $\omega\left(x\right)=\left\{ +\infty\right\} $. For
further discussion on how this observation connects to wild attractors
in unimodal maps, see Bruin and Todd \cite{MR3384890} as well as
\cite{MR3289914}.

\subsection{Gurevich pressure and phase transitions}

We say that $g:\Rep(f)\rightarrow\R$ is Hölder continuous if $g|_{\Rep(f)\cap J_{i}}$
is Hölder continuous with respect to the Euclidean metric and some
exponent $\theta\in\left(0,1\right)$ for each $i\in I$ with a H\"older
constant independent of $i$. In particular, $g|_{\Rep(f)\cap J_{i}}$
is uniformly continuous and has a unique continuous extension to $\overline{\Rep(f)\cap J_{i}}$,
which we denote by $g_{i}$.
\begin{rem*}
If $g:\Rep(f)\rightarrow\R$ Hölder continuous, then $g$ is uniformly
continuous in the following sense: For every $\varepsilon>0$ there
exists $\delta>0$ such that for each $i\in I$ and $x,y\in J_{i}\cap\Rep\left(f\right)$,
\[
\left|x-y\right|<\delta\implies\left|g\left(x\right)-g\left(y\right)\right|<\epsilon.
\]
In fact, for most arguments uniform continuity suffices. Hölder continuity
is essential for the existence of Gibbs measures. For ease of exposition
we restrict to Hölder continuous potentials in the following.
\end{rem*}
The \emph{Gurevich pressure} of a Hölder continuous function $g:\Rep\left(f\right)\rightarrow\R$
with respect to an expanding $C^{2}$-Markov interval map $f$ is
given, for $i\in I$, by
\[
\mathcal{P}\left(g,f\right)\coloneqq\limsup_{n\to\infty}\frac{1}{n}\log\sum_{x\in\Per_{n}\left(f\right)\cap J_{i}}\exp(S_{n}g\left(x\right)),
\]
where $S_{n}g\coloneqq\sum_{k=0}^{n-1}g\circ f^{k}$. This definition
is independent of $i\in I$ (see \prettyref{prop:gurevich_wrt_f}
for the proof and further details).

Adapted to our skew-periodic setting, we now introduce the notion
of \emph{Gurevich pressure with respect to }$\H\in\left\{ 0,\N,\Z\right\} $:
For $g:\Rep_{0}\rightarrow\R$ H\"older continuous, we denote its
\emph{$\H$-periodic extension} by 
\[
g_{\H}:\Rep_{0}+\H\to\R,\,x\mapsto g\left(x-\left\lfloor x\right\rfloor \right)
\]
and we write
\[
\mathcal{P}\left(g,\H\right):=\mathcal{P}\left(g_{\H},\FI_{\H}\right)\;\text{and especially, }\mathcal{P}\left(g\right)\coloneqq\mathcal{P}\left(g,0\right).
\]

The main insight regarding these pressure functions, as established
in \prettyref{thm:fibre_pressure_vs_base_pressure} and \prettyref{thm:Onesided_fibre-pressure_vs_base-pressure},
is captured by the following two \emph{variational formul\ae.} These
are not only crucial for our proofs but also of significant theoretical
interest in their own right:
\begin{align}
\mathcal{P}\left(g,\Z\right) & =\inf_{q\in\R}\mathcal{P}\left(g+q\psi\right),\label{eq:(2)}\\
\mathcal{P}\left(g,\N\right) & =\inf_{q\leq0}\mathcal{P}\left(g+q\psi\right).\label{eq:(3)}
\end{align}
As a partial consequence, we derive the following relations, which
are as well essential for the dimension results:
\begin{align}
\mathcal{P}\left(g,\Z\right)\,\,\leq\,\, & \mathcal{P}\left(g,\N\right)\leq\mathcal{P}\left(g\right),\label{eq:(1)}\\
\mathcal{P}\left(g,\N\right)\,\,=\,\, & \begin{cases}
\mathcal{P}\left(g,\Z\right), & \text{if }\inf_{q\leq0}\mathcal{P}(g+q\psi)=\inf_{q\in\R}\mathcal{P}(g+q\psi),\\
\mathcal{P}(g), & \text{if }\inf_{q>0}\mathcal{P}(g+q\psi)=\inf_{q\in\R}\mathcal{P}(g+q\psi).
\end{cases}\label{eq:(4)}
\end{align}

We call $s\mapsto\mathcal{P}\left(s\varphi,\H\right)$ the \emph{geometric
pressure function} for the skew-periodic map $\FI_{\H}$. It is standard
to verify (see for instance \cite[Theorem 2.1.3]{MR2003772}) that
\begin{equation}
\delta_{\H}=\inf\left\{ s\ge0:\mathcal{P}\left(s\varphi,\H\right)\leq0\right\} .\label{eq:delta_vs_pressure}
\end{equation}
This observation, together with the above pressure formul\ae, is
the key to proving the results of \prettyref{subsec:Dimension-gap},
see \prettyref{subsec:Thermodynamic-formalism} for more details.

Finally, we state our last main result, which concerns the emergence
of a second-order phase transition and its characterisation via an
asymptotic covariance criterion. For this, let $\mu_{g}$ be the unique
 $F_{0}$-invariant Gibbs measure as defined in \prettyref{subsec:Periodic-setting-and}
and recall the notion of \emph{asymptotic covariance }for general
$\varphi,\psi\in L^{2}(\mu_{g})$ given by
\begin{align*}
\cov_{g}\left(\varphi,\psi\right) & \coloneqq\sum_{i=0}^{\infty}\left(\int\varphi\cdot\psi\circ\FI_{0}^{i}\d\mu_{g}-\int\varphi\d\mu_{g}\int\psi\d\mu_{g}\right).
\end{align*}

\begin{thm}[Second-order phase transition]
\label{thm:Phasetransition} For the skew-periodic interval map $\FI_{\N}$
assume there exist $s_{0}>s_{1}>0$ and $q_{0}>0$ such that $\mathcal{P}\left(s_{1}\varphi+q_{0}\psi\right)<\infty$
and $\int\psi\d\mu_{s_{0}\varphi}=0.$ Then, the geometric pressure
function $s\mapsto\mathcal{P}\left(s\varphi,\N\right)$ exhibits a
second-order phase transition at $s_{0}$ if and only if $\cov_{s_{0}\varphi}\left(\varphi,\psi\right)\neq0.$
\end{thm}

The proof will be postponed to \prettyref{subsec:Phase-transition}.

\section{Applications and examples\label{sec:Applications-and-examples}}

\subsection{Reflective simple random walk}

To demonstrate the phenomenon of a second-order phase transition,
we recall one of the simplest examples with only two full branches,
which models a simple random walk as presented in \cite{MR4373996}.
We set up a classical \emph{one-step random walk on $\Z$} using a
skew-periodic interval map, by fixing $c_{1},c_{2}\in\left(0,1\right)$
with $c_{1}+c_{2}\leq1$ and considering the map
\[
F_{0}:x\mapsto\begin{cases}
x/c_{1}, & \text{for }x\in(0,c_{1})\eqqcolon J_{1},\\
\left(x-1+c_{2}\right)/c_{2}, & \text{for }x\in(1-c_{2},1)\eqqcolon J_{2}.
\end{cases}
\]
We define the step function $\Psi_{\Z}$ via the step length function
$\psi\coloneqq-\1_{J_{1}}+\1_{J_{2}}$ and accordingly, $\text{\ensuremath{\Psi_{\N}}}$
with the reflexion rule $\ensuremath{\psi_{1}\coloneqq\psi+1}$. Note
that $\Rep\left(F_{0}\right)$ is a Cantor set with Hausdorff dimension
$\delta_{0}<1$ if and only if $c_{1}+c_{2}<1$, where $\delta_{0}$
is the unique number $s>0$ with $c_{1}^{s}+c_{2}^{s}=1$. Otherwise,
$\Rep\left(F_{0}\right)$ is the unit interval\noun{.}
\begin{description}
\item [{Case~$\boldsymbol{c_{1}>c_{2}}$}] In this case we have $\alpha_{\max}<0$,
and we exhibit the \textbf{lean case}, i.e\@., \newline $\dim_{H}\left(R\left(\FI_{\N}\right)\right)>\dim_{H}\left(T^{+}\left(\FI_{\N}\right)\right)$
and $\mu_{\delta_{0}\varphi}\left(R\left(\FI_{\N}\right)\right)=1$.
\item [{Case~$\boldsymbol{c_{1}=c_{2}}$}] In this case we have $\alpha_{\max}=0$,
and we exhibit the \textbf{balanced case}, i.e\@., \newline $\dim_{H}\left(R\left(\FI_{\N}\right)\right)=\dim_{H}\left(T^{+}\left(\FI_{\N}\right)\right)$
and $\mu_{\delta_{0}\varphi}\left(R\left(\FI_{\N}\right)\right)=1$.
\item [{Case~$\boldsymbol{c_{1}<c_{2}}$}] In this case we have $\alpha_{\max}>0$,
and we exhibit the \textbf{black hole case}, i.e\@., \newline we
have a dimension drop $\dim_{H}\left(R\left(\FI_{\N}\right)\right)<\dim_{H}\left(T^{+}\left(\FI_{\N}\right)\right)$
and $\mu_{\delta_{0}\varphi}\left(T^{+}\left(\FI_{\N}\right)\right)=1.$
\end{description}
This family of examples also results in systems with and without a
second-order phase transition for $F_{\N}$, as we demonstrate next:
The condition $0=\int\psi\d\mu_{s\varphi+q(s)\psi}=-\mu_{s\varphi+q(s)\psi}\left(J_{1}\right)+\mu_{s\varphi+q(s)\psi}\left(J_{2}\right)$
is equivalent to 
\[
c_{2}^{s}\e^{q\left(s\right)}=c_{1}^{s}\e^{-q\left(s\right)}\iff s\left(\log c_{1}-\log c_{2}\right)/2=q\left(s\right),
\]
providing us for all $s\in\R$ with a solution $q\left(s\right)$
for $0=\int\psi\d\mu_{s\varphi+q(s)\psi}$. Further, for $s_{0}=0$
and $c_{1},c_{2}\in\left(0,1\right)$ arbitrary, or $c_{1}=c_{2}$
and $s_{0}$ arbitrary, we have $q(s_{0})=0$.

\emph{Case 1}: For $c_{1}=c_{2}$ we have $q(s_{0})=0$ for all $s_{0}\in\R$.
Furthermore, $\varphi$ is constant such that for all $s_{0}\in\R$,
\[
\cov_{s_{0}\varphi}\left(\varphi,\psi\right)=0.
\]
 Hence, by \prettyref{thm:Phasetransition} there is no phase transition
in this case--in fact, $s\mapsto\mathcal{P}\left(s\varphi,\N\right)$
defines a straight line.

\emph{Case 2}: For $c_{1}\neq c_{2}$ we have $q\left(s_{0}\right)=0$
only for $s_{0}=0$. Hence, in this case $\mu_{0\varphi}$ is the
$\left(1/2,1/2\right)$-Bernoulli measure and 
\begin{align*}
\cov_{0\varphi}\left(\varphi,\psi\right) & =\frac{-1}{2}\left(\frac{1}{2}\left(\log c_{1}-\log c_{2}\right)\right)+\frac{1}{2}\left(\frac{1}{2}\left(-\log c_{1}+\log c_{2}\right)\right)\\
 & =\frac{1}{2}\left(\log c_{2}-\log c_{1}\right)\neq0
\end{align*}
and by \prettyref{thm:Phasetransition} there necessarily is a second-order
phase transition at $s_{0}=0$. In fact, $\mathcal{P}(s\varphi+q\psi)=\log\left(c_{1}^{s}\e^{-q}+c_{2}^{s}\e{}^{q}\right)$,
$\mathcal{P}(s\varphi)=\log\left(c_{1}^{s}+c_{2}^{s}\right)$ and
$\inf_{q}\mathcal{P}(s\varphi+q\psi)=\log\left(2\left(c_{1}c_{2}\right)^{s/2}\right).$
By combining the variational formul\ae~\prettyref{eq:(2)} and \prettyref{eq:(3)}
with \prettyref{eq:(4)}, we have
\[
\mathcal{P}\left(s\varphi,\N\right)=\begin{cases}
\mathcal{P}\left(s\varphi,\Z\right), & \text{if }c_{1}^{s}\leq c_{2}^{s},\\
\mathcal{P}(s\varphi), & \text{if }c_{1}^{s}>c_{2}^{s},
\end{cases}
\]
and, assuming with no loss of generality $c_{1}<c_{2}$, we particularly
find
\[
\mathcal{P}\left(s\varphi,\N\right)=\begin{cases}
\log2+s/2\cdot\log\left(c_{1}c_{2}\right), & \text{if }s\geq0,\\
\log\left(c_{1}^{s}+c_{2}^{s}\right), & \text{if }s<0.
\end{cases}
\]
This confirms that in this case, a second-order phase transition occurs
at $s_{0}=0$.

\subsection{$\mathfrak{a}$-Lüroth maps\label{subsec:vanstrien}}

The following examples are intended to cover earlier work initialised
by van Strien and worked out mainly by Stratmann and Vogt \cite{MR1438267}
as well as Bruin and Todd \cite{MR2959300,MR3384890}.

In this and the following subsection, let us consider two examples
of $\mathfrak{a}$-Lüroth maps as introduced in \cite{MundayKessStrat10}.

First, we consider the family of $\mathfrak{a}(\lambda)$-\emph{Lüroth
maps} $L_{\lambda}$, $\lambda\in\left(0,1\right)$ which are given
by the countable partitions $\mathfrak{a}(\lambda)=\left\{ J_{k}\coloneqq\left(\lambda^{k},\lambda^{k-1}\right):k\ge1\right\} $
as follows, for $x\in J_{n}$,
\[
L_{\lambda}\left(x\right)=\left(x-t_{n}\right)/\left|J_{n}\right|,
\]
where $t_{n}\coloneqq\sum_{k=n+1}^{\infty}\left|J_{k}\right|$ and
$\left|J_{k}\right|$ refers to the length of the interval $J_{k}$.
Then the step function of the non-reflective $\mathfrak{a}\left(\lambda\right)$-Lüroth
system $L_{\lambda,\Z}$ is given by the step length function $\psi\left(x\right)\coloneqq k-2$,
for $x\in J_{k},$ $k\ge1$. For the reflective $\mathfrak{a}\left(\lambda\right)$-Lüroth
system $L_{\lambda,\N}$ we also choose $\psi_{1}\coloneqq\psi+1$.
The graphs of the skew-periodic interval maps $L_{\lambda,\Z}$ and
$L_{\lambda,\N}$ are given in \prettyref{fig:Lueroth-system-non-refelctive+reflective-1}.
We will use the short-hand notation $\delta_{\H}(\lambda)\coloneqq\delta\left(L_{\lambda,\H}\right)$
for $\H\in\left\{ 0,\N,\Z\right\} $.

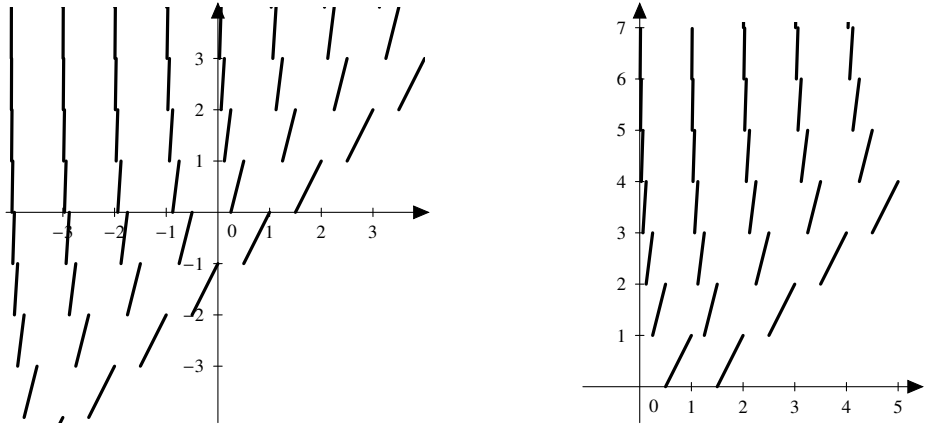
\begin{figure}[H]
\begin{tikzpicture}[scale=.8,line cap=round,line join=round,>=triangle 45,x=.85cm,y=0.85cm] \draw[->,color=black] (-4.1,0.) -- (4.1,0.); \foreach \x in {-3,-2,-1,1,2,3} \draw[shift={(\x,0)},color=black] (0pt,2pt) -- (0pt,-2pt) node[below] {\scriptsize $\x$}; \draw[->,color=black] (0.,-4.1) -- (0.,4.1); \foreach \y in {-3,-2,-1,1,2,3} \draw[shift={(0,\y)},color=black] (2pt,0pt) -- (-2pt,0pt) node[left] {\scriptsize $\y$}; \draw[color=black] (0pt,-9pt) node[right] {\scriptsize $0$}; \clip(-4.1,-4.1) rectangle (4.0,4.0);

\foreach \k in {-4,-3,-2,-1,0,1,2,3,4 }
\foreach \m in {1,2,3,4,5,6,7,8,9 }
\draw[line width=1.2pt, samples=3,domain={2^(-\m)+\k}:{2^(-\m+1)+\k}] 
plot(\x,{2^\m*(\x-\k)+\m+\k-3});

  \end{tikzpicture}\hspace*{2cm}\begin{tikzpicture}[scale=.8,line cap=round,line join=round,>=triangle 45,x=.85cm,y=0.85cm] \draw[->,color=black] (-1.1,0.) -- (5.5,0.); \foreach \x in {1,2,3,4,5} \draw[shift={(\x,0)},color=black] (0pt,2pt) -- (0pt,-2pt) node[below] {\scriptsize $\x$}; \draw[->,color=black] (0.,-0.7) -- (0.,7.5); \foreach \y in {1,2,3,4,5,6,7} \draw[shift={(0,\y)},color=black] (2pt,0pt) -- (-2pt,0pt) node[left] {\scriptsize $\y$}; 
\draw[color=black] (0pt,-9pt) node[right] {\scriptsize $0$}; 
\clip(-.1,-0.1) rectangle (5.5,7.1); 

\foreach \m in {1,2,3,4,5,6,7 }
\draw[line width=1.2pt, samples=10,domain={2^(-\m)}:{2^(-\m+1)}] 
plot(\x,{2^\m*(\x)+\m-2});  
\foreach \k in {1,2,3,4 }
\foreach \m in {1,2,3,4,5,6,7 }
\draw[line width=1.2pt, samples=10,domain={2^(-\m)+\k}:{2^(-\m+1)+\k}] 
plot(\x,{2^\m*(\x-\k)+\m+\k-3}); 
  
\end{tikzpicture}

\caption{To the left the graph of $L_{\lambda,\Z}$ modelling the non-reflective
$\mathfrak{a}(\lambda)$-Lüroth system and to the right the graph
of $L_{\lambda,\N}$ modelling the reflective $\mathfrak{a}(\lambda)$-Lüroth
system.\label{fig:Lueroth-system-non-refelctive+reflective-1}}
\end{figure}

The reflective $\mathfrak{a}\left(\lambda\right)$-Lüroth map $L_{\lambda,\N}$
is naturally conjugate to a dynamical system $V=V_{\lambda}$ on $(0,1)$
by the conjugacy
\[
p:(0,1)\to\Rep\left(L_{\lambda,\N}\right),\,x\mapsto\frac{\lambda^{-k+1}}{1-\lambda}(x-\lambda^{k})+k-1,\textnormal{ for }x\in J_{k},\,k\ge1.
\]
This family of systems $\left(V_{\lambda}:\lambda\in(0,1)\right)$,
which is given by 
\begin{align*}
V_{\lambda}=V:(0,1) & \to(0,1),\,x\mapsto\begin{cases}
\frac{{x-\lambda}}{1-\lambda}, & x\in(\lambda,1),\\
\frac{{x-\lambda^{k+1}}}{\lambda(1-\lambda)}, & x\in(\lambda^{k+1},\lambda^{k}),\,k\ge1,
\end{cases}
\end{align*}
has been studied in \cite{MR1438267,MR2959300,MR3610938} and the
Hausdorff dimension of the transient and the recurrent part has explicitly
been determined therein. In fact, by the conjugation there is a one-to-one
correspondence between the recurrent set of $L_{\lambda,\N}$ and
the \emph{recurrent set} of $V_{\lambda}$ given by 
\[
\mathbf{R}_{\lambda}\coloneqq\left\{ x\in(0,1):\exists k>1\text{ with }V_{\lambda}^{n}(x)\geq1/k\text{ for infinitely many }n\in\N\right\} 
\]
and, similarly, for the \emph{transient set} given by 
\[
\mathbf{T}_{\lambda}\coloneqq\left\{ x\in(0,1)\colon V_{\lambda}^{n}(x)\to0\right\} .
\]
Then our formalism reproves also the fact that
\begin{equation}
\dim_{H}(\mathbf{R}_{\lambda})=\dim_{H}(\mathbf{T}_{1-\lambda})=\begin{cases}
1, & \lambda\le1/2,\\
\frac{-\log4}{\log\left(\lambda\left(1-\lambda\right)\right)}, & \lambda>1/2,
\end{cases}\label{eq:bt}
\end{equation}
which is one of the main results obtained in \cite{MR1438267,MR2959300}.
See \prettyref{fig:Lambda-Transience_Recurrent_Lueroth-Spectrum}
for the graph of this function and for the conjugate system see \prettyref{fig:Conjugacy_Gauss+Lueroth-system_reflective}.
\begin{figure}[h]
\pgfplotsset{width=9cm,compat=1.9}
\begin{tikzpicture}[line cap=round,line join=round,>=triangle 45]
\begin{axis}[ymin=-0.1,ymax=1.1,xmin=0,xmax=1.1,     axis equal,     axis lines=middle,     axis line style={->},     tick style={color=black,line width=1.1pt},     xtick={0,0.5,1}, xticklabels={$0$,$1/2$,$1$}, ytick={0,1},yticklabels={$0$, $1$} ]     
\addplot 
[         domain=0:0.5, samples=497,line width=1.2pt, dashed 
] {-ln(4.0)/ ln(x*(1-x))}; 
\addplot 
[         domain=0:0.5, samples=20,line width=1.2pt,   
] {1};
\addplot 
[         domain=0.5:1, samples=497,line width=1.2pt,   
] {-ln(4.0)/ ln(x*(1-x))};
\addplot 
[         domain=0.5:1, samples=20,line width=1.2pt, dashed   
] {1};
\addplot 
[         domain=0:0.15, samples=20,line width=1.2pt,   
] (1,x);
\addplot [        dashed, domain=0:0.19, samples=20,line width=1.2pt,   
] (0,x);
 
\end{axis}
 \end{tikzpicture}\caption{The dimensions of the transient part $\lambda\protect\mapsto\dim_{H}\left(T^{+}\left(L_{\lambda,\N}\right)\right)=\dim_{H}\left(\mathbf{T}_{\lambda}\right)$
(dashed line) and recurrent part $\lambda\protect\mapsto\dim_{H}\left(R\left(L_{\lambda,\N}\right)\right)=\dim_{H}\left(\mathbf{R}_{\lambda}\right)$
(solid line) as a function of the parameter $\lambda$ for the Lüroth
system considered in \prettyref{subsec:vanstrien}. \label{fig:Lambda-Transience_Recurrent_Lueroth-Spectrum}}
\end{figure}
 To derive \eqref{eq:bt} from our theory, first note that $\dim_{H}(\mathbf{R}_{\lambda})=\dim_{H}\left(R\left(L_{\lambda,\N}\right)\right)$
and $\dim_{H}(\mathbf{T}_{\lambda})=\dim_{H}\left(T^{+}\left(L_{\lambda,\N}\right)\right)$
because $p|_{J_{k}}:J_{k}\rightarrow(k,k+1)$ is bi-Lipschitz for
each $k\ge1$. The geometric potential of $L_{\lambda}$ is given,
for $x\in J_{k}$, by 
\[
\varphi(x):=\varphi_{\lambda}\left(x\right)\coloneqq\log\left(\left(1-\lambda\right)\lambda^{k-1}\right).
\]
Clearly, we have $\delta_{0}:=\delta_{0}(\lambda)=1$ for all $\lambda\in(0,1)$.
For $s,q\in\R$ we have
\begin{align}
z_{\lambda}(s,q):= & \exp\left(\mathcal{P}\left(s\varphi_{\lambda}+q\psi\right)\right)=(1-\lambda)^{s}\left(\sum_{k=1}^{\infty}\lambda^{(k-1)s}\e^{q(k-2)}\right)=\begin{cases}
\frac{(1-\lambda)^{s}\e^{-q}}{1-\lambda^{s}\e^{q}}, & \e^{q}<\lambda^{-s},\\
+\infty, & \e^{q}\ge\lambda^{-s.}
\end{cases}\label{eq:bt2}
\end{align}
Now, suppose that $\e^{q}<\lambda^{-s}$. Then we have
\begin{equation}
\int\psi\d\mu_{s\varphi_{\lambda}+q\psi}=\frac{\partial z_{\lambda}}{\partial q}(s,q)=-\frac{(1-\lambda)^{s}\left(\e^{-q}-2\lambda^{s}\right)}{\left(1-\lambda^{s}\e^{q}\right)^{2}}.\label{eq:bt3}
\end{equation}
In particular, for $s=1$ and $q=0$ we have 
\[
\alpha_{\max}(\lambda)=\int\psi\d\mu_{\delta_{0}(\lambda)\varphi_{\lambda}}=\frac{2\lambda-1}{1-\lambda}
\]
and by \prettyref{prop:Bowen_for-the-skew-periodic-1} and \prettyref{thm:Big_Theorem}
we can deduce the following trichotomy.
\begin{description}
\item [{Case~$\boldsymbol{\lambda<1/2}$}] In this case we have $\alpha_{\max}\left(\lambda\right)<0$,
and we exhibit the \textbf{lean case}, i.e\@., \newline $1=\dim_{H}\left(R\left(L_{\lambda,\N}\right)\right)>\dim_{H}\left(T^{+}\left(L_{\lambda,\N}\right)\right)=\delta_{\Z}\left(\lambda\right)$
and $\mu_{\delta_{0}\varphi}\left(R\left(L_{\lambda,\N}\right)\right)=1$.
\item [{Case~$\boldsymbol{\lambda=1/2}$}] In this case we have $\alpha_{\max}\left(\lambda\right)=0$,
and we exhibit the \textbf{balanced case}, i.e\@., \newline $1=\dim_{H}\left(R\left(L_{\lambda,\N}\right)\right)=\dim_{H}\left(T^{+}\left(L_{\lambda,\N}\right)\right)$
and $\mu_{\delta_{0}\varphi}\left(R\left(L_{\lambda,\N}\right)\right)=1$.
\item [{Case~$\boldsymbol{\lambda>1/2}$}] In this case we have $\alpha_{\max}\left(\lambda\right)>0$,
and we exhibit the \textbf{black hole case}, i.e\@., \newline $\delta_{\Z}(\lambda)=\delta_{\N}(\lambda)=\dim_{H}\left(R\left(L_{\lambda,\N}\right)\right)<\dim_{H}\left(T^{+}\left(L_{\lambda,\N}\right)\right)=1$,
we have a dimension drop, and $\mu_{\delta_{0}\varphi}\left(T^{+}\left(L_{\lambda,\N}\right)\right)=1.$
\end{description}
Note that a dimension gap occurs if and only if $\lambda>1/2$ (cf.
\prettyref{thm:Big_Theorem}). We finally verify that 
\begin{equation}
\delta_{\Z}(\lambda)=\frac{-\log4}{\log\left(\lambda\left(1-\lambda\right)\right)}.\label{eq:bt1}
\end{equation}
By Theorem \ref{thm:fibre_pressure_vs_base_pressure} (1), the unique
solution $(s_{0},q_{0})$ of $z_{\lambda}(s,q)=1$ and  $\int\psi\d\mu_{s\varphi_{\lambda}+q\psi}=0$
is of the form $s_{0}=\delta_{\Z}(\lambda)$. By \eqref{eq:bt3} we
have $\e^{q_{0}}=\lambda^{-s_{0}}/2$. Substituting this into $z_{\lambda}(s,q)=1$
and using \eqref{eq:bt2}, gives $(1-\lambda)^{s_{0}}2\lambda^{s_{0}}=1/2$.
Since $\delta_{\Z}(\lambda)=\delta_{\Z}(1-\lambda)$, the proof of
\eqref{eq:bt} is complete.

Next, we apply \prettyref{thm:Phasetransition} to derive that $s\mapsto\mathcal{P}\left(s\varphi,\N\right)$
has a second-order phase transition at $s_{0}=-\log2/\log\lambda$,
reproving another main result of \cite{MR2959300}. For each $\lambda\in(0,1)$
we have that $s_{0}=-\log2/\log\lambda$ solves $q\left(s_{0}\right)=0$.
For $n\ge1$, we have 
\[
\mu_{s_{0}\varphi}\left(J_{n}\right)=\frac{(1-\lambda)^{s_{0}}\lambda^{(n-1)s_{0}}}{\sum_{k\ge1}(1-\lambda)^{s_{0}}\lambda^{(k-1)s_{0}}}=\frac{\lambda^{(n-1)s_{0}}}{\sum_{k\ge1}\lambda^{(k-1)s_{0}}}=\frac{2^{-(n-1)}}{\sum_{k\ge1}2^{-(k-1)}}=2^{-n}
\]
and consequently, 
\[
\int\varphi\d\mu_{s_{0}\varphi}=\sum_{n\ge1}\log((1-\lambda)\lambda^{n-1})2^{-n}=\log((1-\lambda)\lambda),
\]
and 
\begin{align*}
\cov_{s_{0}\varphi}\left(\varphi,\psi\right) & =\int(\varphi-\mu_{s_{0}\varphi}(\varphi))\cdot(\psi-\mu_{s_{0}\varphi}(\psi))\d\mu_{s_{0}\varphi}\\
 & =\sum_{n\ge1}\left(\log\left(\left(1-\lambda\right)\lambda^{n-1}\right)-\log((1-\lambda)\lambda)\right)(n-2)2^{-n}\\
 & =\sum_{n\ge1}\left((n-2)\log\lambda\right)(n-2)2^{-n}=2\log\lambda\neq0.
\end{align*}
Since  $\mathcal{P}(s\varphi+q\psi)$ is finite in a neighbourhood
of $(s_{0},0)$ by \prettyref{eq:bt2}, \prettyref{thm:Phasetransition}
implies the existence of a second-order phase transition of $s\mapsto\mathcal{P}\left(s\varphi,\N\right)$
at $s_{0}$.

\subsection{Drift without dimension gap for non-reflective system\label{subsec:Drift-without-dimension}}

This example demonstrates that the second condition in \prettyref{thm:DimGap_non-reflective}
is meaningful for the absence of a dimension gap. For this we consider
again an $\mathfrak{a}$-Lüroth map this time with partition $\mathfrak{a}=\left(J_{k}:k\in\N\right)$
given by $J_{0}:=(a,1)$, $J_{k}:=\left(a/\left(k+1\right),a/k\right)$,
$k\geq1$, for some fixed $a\in\left(0,1\right)$ and step length
function $\psi(x)\coloneqq\left\lfloor \sqrt{k}\right\rfloor -1$
for $x\in J_{k}$.  Then $\delta_{0}=1$ and $\mu_{\delta_{0}\varphi}$
is equal to the Lebesgue measure restricted to the unit interval.
A straightforward calculation shows that if $a$ is chosen sufficiently
close to $0$, we have $-1<\alpha_{\max}<0$ and simultaneously, for
all $q>0$,
\[
\sum_{\FI_{0}(x)=x}\exp\left(\delta_{0}\varphi\left(x\right)+q\psi\left(x\right)\right)=+\infty.
\]
 Hence, by \prettyref{thm:DimGap_non-reflective}, we have no dimension
gap despite the presence of a negative drift.

\subsection{A non-linear system--the Gauss system}

\begin{figure}[H]
\begin{tikzpicture}[scale=.8, line cap=round,line join=round,>=triangle 45,x=.85cm,y=0.85cm] \draw[->,color=black] (-4.1,0.) -- (4.1,0.); \foreach \x in {-3,-2,-1,1,2,3} \draw[shift={(\x,0)},color=black] (0pt,2pt) -- (0pt,-2pt) node[below] {\scriptsize $\x$}; \draw[->,color=black] (0.,-4.1) -- (0.,4.1); \foreach \y in {-3,-2,-1,1,2,3} \draw[shift={(0,\y)},color=black] (2pt,0pt) -- (-2pt,0pt) node[left] {\scriptsize $\y$}; \draw[color=black] (0pt,-9pt) node[right] {\scriptsize $0$}; \clip(-4.1,-4.1) rectangle (4.1,4.1); 
  
\draw[line width=1.2pt, samples=100,domain=-2.9:-2] 
plot(\x,{1/(\x+3)-5}); 
\draw[line width=1.2pt, samples=100,domain=-1.9:-1] 
plot(\x,{1/(\x+2)-4}); 

\draw[line width=1.2pt, samples=100,domain=-0.9:0] 
plot(\x,{1/(\x+1)-3}); \draw[line width=1.2pt, samples=100,domain=0.1:1] 
plot(\x,{1/(\x)-2}); \draw[line width=1.2pt, samples=100,domain=1.1:2] 
plot(\x,{1/(\x-1)-1}); \draw[line width=1.2pt, samples=100,domain=2.1:3] 
plot(\x,{1/(\x-2)}); \draw[line width=1.2pt, samples=100,domain=3.1:4] 
plot(\x,{1/(\x-3)+1});

  \end{tikzpicture}\hspace*{2cm}\begin{tikzpicture}[scale=.8,line cap=round,line join=round,>=triangle 45,x=.85cm,y=0.85cm] \draw[->,color=black] (-1.1,0.) -- (5.5,0.); \foreach \x in {1,2,3,4,5} \draw[shift={(\x,0)},color=black] (0pt,2pt) -- (0pt,-2pt) node[below] {\scriptsize $\x$}; \draw[->,color=black] (0.,-0.7) -- (0.,7.5); \foreach \y in {1,2,3,4,5,6,7} \draw[shift={(0,\y)},color=black] (2pt,0pt) -- (-2pt,0pt) node[left] {\scriptsize $\y$}; 
\draw[color=black] (0pt,-9pt) node[right] {\scriptsize $0$}; 
\clip(-.1,-0.1) rectangle (5.5,7.1);

\draw[line width=1.2pt, samples=100,domain=0.1:1] 
plot(\x,{1/(\x)-1}); 
\draw[line width=1.2pt, samples=100,domain=1.1:2] 
plot(\x,{1/(\x-1)-1}); 
\draw[line width=1.2pt, samples=100,domain=2.1:3] 
plot(\x,{1/(\x-2)}); 
\draw[line width=1.2pt, samples=100,domain=3.1:4] 
plot(\x,{1/(\x-3)+1}); 
\draw[line width=1.2pt, samples=100,domain=4.1:5] 
plot(\x,{1/(\x-4)+2});
 
\end{tikzpicture}

\caption{To the left the graph of $G_{\Z}$ modelling the non-reflective Gauss
system and to the right $G_{\N}$ modelling the reflective Gauss system.\label{fig:Gauss-system-non-refelctive_and_reflective}}
\end{figure}
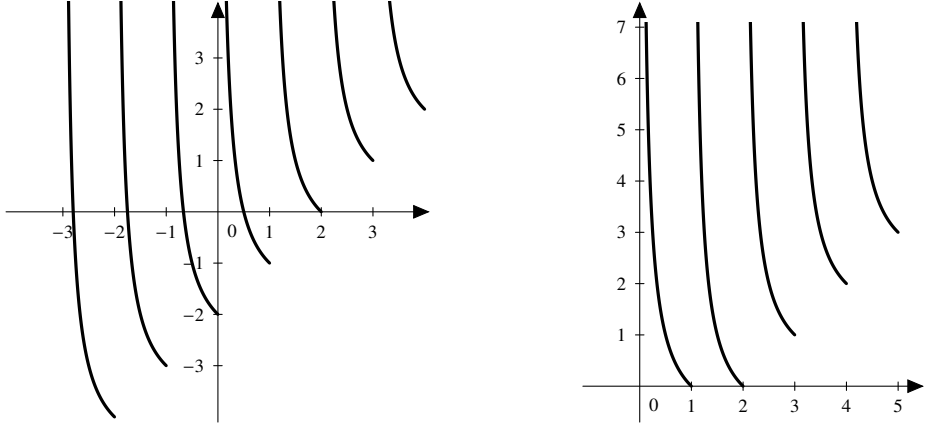
 As a non-linear example with infinitely many branches, we consider
the \emph{Gauss map} $G:(0,1)\to(0,1)$, $x\mapsto1/x-\left\lfloor 1/x\right\rfloor $
and let $J_{k}\coloneqq\left(\frac{1}{k+1},\frac{1}{k}\right)$, $k\geq1$.
The step length function is chosen to be $\psi\left(x\right)\coloneqq k-2$,
for $x\in J_{k},\;k\geq1$. For the reflective Gauss system $G_{\N}$
we choose $\psi_{1}\coloneqq\psi+1$. The graphs of the skew-periodic
interval maps $G_{\Z}$ and $G_{\N}$ are given in \prettyref{fig:Gauss-system-non-refelctive_and_reflective}.

We have $\delta_{0}=1$ and, since $\mu_{\delta_{0}\varphi}$ is the
Gauss measure, we obtain 
\[
\alpha_{\max}=\int\psi\d\mu_{\delta_{0}\varphi}=\frac{\log3}{\log2}-2+\frac{1}{\log2}\sum_{k=3}^{\infty}\left(k-2\right)\log\left(1+\frac{1}{k\left(k+2\right)}\right)=+\infty.
\]
By \prettyref{thm:Big_Theorem}, we obtain that $\delta\left(G_{\N}\right)=\delta\left(G_{\Z}\right)$.
By \prettyref{thm:DimGap_non-reflective}, we have $\delta\left(G_{\Z}\right)<1$.
Accordingly, by \prettyref{prop:Bowen_for-the-skew-periodic-1}, we
get 
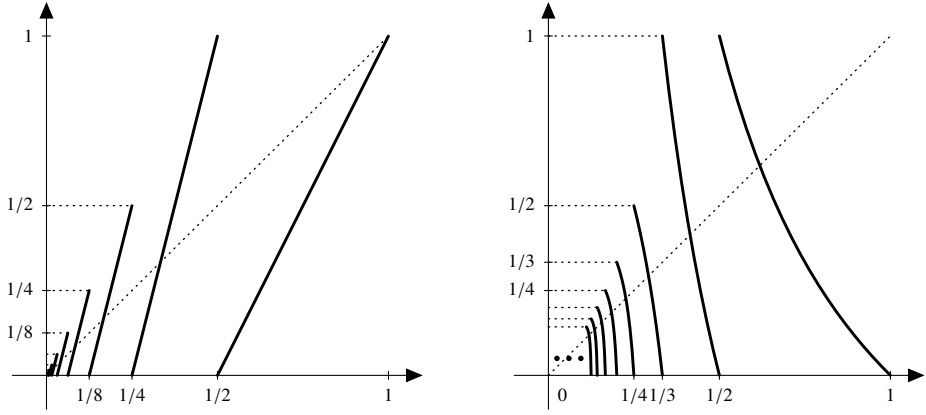
\begin{figure}[h]
\begin{tikzpicture}[scale=.75,line cap=round,line join=round,>=triangle 45,x=6cm,y=6cm] \draw[->,color=black] (-.1,0.) -- (1.1,0.); \foreach \x in {1/8,1/4,1/2,1} \draw[shift={(\x,0)},color=black] (0pt,2pt) -- (0pt,-2pt) node[below] {\scriptsize $\x$}; \draw[->,color=black] (0.,-0.1) -- (0.,1.1); \foreach \y in {1/8,1/4,1/2,1} \draw[shift={(0,\y)},color=black] (2pt,0pt) -- (-2pt,0pt) node[left] {\scriptsize $\y$}; 
\clip(-0.1,-0.1) rectangle (1.2,1.2); 
  
\draw[line width=1.1pt, samples=20,domain=1/2:1] 
plot(\x,{2*\x-1}); 
 
\foreach \k in {1,2,3,4,5,6,7 }
\draw[line width=1.1pt, samples=10,domain=2^(-\k-1):2^(-\k)] 
plot(\x,{4*\x-2^(-\k+1)}); 

\foreach \k in {1,2,3,4,5,6,7 }
\draw[line width=0.5pt,dotted, samples=2,domain=0:2^(-\k-1)] 
plot(\x,{2^(-\k)}); 

\draw[line width=.5pt,dotted , samples=100,domain=0:1] 
plot(\x,{ \x  )});
 

  \end{tikzpicture}~\hspace*{\fill}\begin{tikzpicture}[scale=.75,line cap=round,line join=round,>=triangle 45,x=6cm,y=6cm] \draw[->,color=black] (-.1,0.) -- (1.1,0.); \foreach \x in {1/4,1/3,1/2,1} \draw[shift={(\x,0)},color=black] (0pt,2pt) -- (0pt,-2pt) node[below] {\scriptsize $\x$}; \draw[->,color=black] (0.,-0.1) -- (0.,1.1); \foreach \y in {1/4,1/3,1/2,1} \draw[shift={(0,\y)},color=black] (2pt,0pt) -- (-2pt,0pt) node[left] {\scriptsize $\y$}; \draw[color=black] (0pt,-10pt) node[right] {\scriptsize $0$}; \clip(-0.1,-0.1) rectangle (1.2,1.2); 
  
\draw[line width=1.1pt, samples=20,domain=1/2:1] 
plot(\x,{(1-\x)/\x}); 
 
\foreach \k in {1,2,3,4,5,6,7 }
\draw[line width=1.1pt, samples=10,domain=1/(\k+2):1/(\k+1)] 
plot(\x,{(1-(\k+1)*\x)/((2-\k^2)*\x+\k-1}); 
\foreach \k in {1,2,3,4,5,6,7 }
\draw[line width=0.5pt,dotted, samples=2,domain=0:1/(\k+2)] 
plot(\x,{1/\k}); 

\draw[line width=.5pt,dotted , samples=100,domain=0:1] 
plot(\x,{ \x  )});
 
\foreach \k in {1,2,3} \filldraw (\k*0.035-0.01,0.05) circle (1.1pt);

  \end{tikzpicture}

\caption{\label{fig:Conjugacy_Gauss+Lueroth-system_reflective} On the left
we have the first eight branches of the conjugate system $V_{\lambda}$
to the right reflective $\mathfrak{a}\left(\lambda\right)$-Lüroth
map $L_{\lambda,\N}$ with $\lambda=1/2$. On the right the corresponding
graph for the system conjugate to the Gauss system $G_{\N}$ with
conjugacy map given by $p:(0,1)\to\R_{>0}$, $x\protect\mapsto1/x-1$. }
\end{figure}
\[
\dim_{H}(R_{u}(G_{\N}))=\dim_{H}(R(G_{\N}))=\dim_{H}(R_{u}(G_{\Z}))=\dim_{H}(R(G_{\Z}))=\delta\left(G_{\N}\right)=\delta\left(G_{\Z}\right)<1,
\]
and a dimension gap occurs for both, $G_{\Z}$ and $G_{\N}$. For
the transient sets, \prettyref{thm:HD_Transient} together with the
analyticity of the function $\left(s,t\right)\mapsto\mathcal{P}\left(s\phi+t\psi\right)$
implies 
\[
\dim_{H}(T^{+}(G_{\N}))=\dim_{H}(T^{+}(G_{\Z}))=1>\dim_{H}(T^{-}(G_{\Z}))>\dim_{H}(T^{-}(G_{\N}))=0.
\]
Also, $\mu_{\delta_{0}\varphi}\left(T^{+}\left(G_{\N}\right)\right)=1$
and we exhibit the \textbf{black hole }case. Analogously to the Lüroth
system, we illustrate a conjugate system acting on the unit interval
in \prettyref{fig:Conjugacy_Gauss+Lueroth-system_reflective}.

Concerning the possibility of a phase transition, note that \prettyref{thm:Phasetransition}
does not apply, as the required pressure condition is not fulfilled.
Indeed, the pressure equals $+\infty$ for all $q>0$. Nevertheless,
we expect the covariance criterion to hold also for this example;
however, perturbation theory at the boundary of the domain of holomorphy
of the Perron-Frobenius operator will be necessary. Elementary calculations,
however, allow us to prove the existence of a phase transition in
a situation very similar to that of the Gauss map if we consider a
linearised version of the latter. The following example illustrates
this.

\subsection{The linearised Gauss system}

This example provides a situation where \prettyref{thm:Phasetransition}
does not apply but our criterion holds nevertheless. As a linearised
Gauss system we consider the $\alpha_{G}$-Lüroth map $L_{\alpha_{G}}$
where, for $k\geq1$, the length of the interval $J_{k}$ is given
by 
\[
\frac{1}{k^{2}\zeta\left(2\right)}.
\]
As before, $\delta_{0}=1$ and $\mu_{\delta_{0}\varphi}$ is the Lebesgue
measure. If we choose for the step length function $\psi\left(x\right)\coloneqq k-2$
for $x\in J_{k}$, then we obtain 
\[
\alpha_{\max}=\int\psi\d\mu_{\delta_{0}\varphi}=\frac{1}{\zeta\left(2\right)}\sum_{k=1}^{\infty}\frac{k-2}{k^{2}}=+\infty.
\]
For the reflective $\alpha_{G}$-Lüroth system $L_{\alpha_{G},\N}$
we choose $\psi_{1}\coloneqq\psi+1$. By \prettyref{thm:Big_Theorem},
we obtain that $\delta\left(L_{\alpha_{G},\N}\right)=\delta\left(L_{\alpha_{G},\Z}\right)$.
By \prettyref{thm:DimGap_non-reflective}, we have $\delta\left(L_{\alpha_{G},\Z}\right)<1$.
Therefore, by \prettyref{prop:Bowen_for-the-skew-periodic-1}, we
have 
\begin{align*}
\dim_{H}\left(R_{u}\left(L_{\alpha_{G},\N}\right)\right) & =\dim_{H}\left(R\left(L_{\alpha_{G},\N}\right)\right)=\dim_{H}\left(R_{u}\left(L_{\alpha_{G},\Z}\right)\right)=\dim_{H}\left(R\left(L_{\alpha_{G},\Z}\right)\right)\\
 & =\delta\left(L_{\alpha_{G},\N}\right)=\delta\left(L_{\alpha_{G},\Z}\right)<1,
\end{align*}
and a dimension gap occurs for both, $L_{\alpha_{G},\Z}$ and $L_{\alpha_{G},\N}$.
For the transient sets, \prettyref{thm:HD_Transient} implies 
\[
\dim_{H}\left(T^{+}\left(L_{\alpha_{G},\N}\right)\right)=\dim_{H}\left(T^{+}\left(L_{\alpha_{G},\Z}\right)\right)=1>\dim_{H}\left(T^{-}\left(L_{\alpha_{G},\Z}\right)\right)>\dim_{H}\left(T^{-}\left(L_{\alpha_{G},\N}\right)\right)=0.
\]
Again, $\mu_{\delta_{0}\varphi}\left(T^{+}\left(L_{\alpha_{G},\N}\right)\right)=1$
and we exhibit the \textbf{black hole }case.

Furthermore, for $s>1$ the Gibbs measure $\mu_{s\varphi}$ is given
on the symbolic level by the Bernoulli measure with probability vector
\[
p_{k}\coloneqq\frac{1}{k^{2s}\zeta\left(2s\right)}.
\]
Hence,
\[
\int\psi\d\mu_{s\varphi}=\frac{\zeta\left(2s-1\right)}{\zeta\left(2s\right)}-2.
\]

Since 
\begin{equation}
\cov_{\mu_{s\varphi}}\left(\varphi,\psi\right)=\frac{\partial}{\partial s}\int\psi\d\mu_{s\varphi}=\frac{\partial}{\partial s}\frac{\zeta\left(2s-1\right)}{\zeta\left(2s\right)}=2\frac{\zeta'\left(2s-1\right)\zeta\left(2s\right)-\zeta\left(2s-1\right)\zeta'\left(2s\right)}{\zeta\left(2s\right)^{2}}<0,\label{eq:cov-neg}
\end{equation}
the function $s\mapsto\int\psi\d\mu_{s\varphi}$ is strictly decreasing
on $s>1$ and negative at $s=2$, hence has a unique zero in $s_{0}>1$
(in fact, $s_{0}\in\left(1.23,1.24\right)$). 

By the dominated convergence theorem we have that 
\[
(-\infty,\infty)\times(-\infty,0)\ni(s,q)\mapsto\frac{\partial}{\partial q}\mathcal{P}(s\varphi+q\psi)=\int\psi\d\mu_{s\varphi+q\psi}=\e^{-\mathcal{P}(s\varphi+q\psi)}\sum_{k=1}^{\infty}(k-2)k^{-2s}\e^{q(k-2)}
\]
is continuous. Since for any $s<s_{0}$ we have that $q\mapsto\sum_{k=1}^{\infty}(k-2)k^{-2s}\e^{q(k-2)}$
is strictly monotone increasing with $\sum_{k=1}^{\infty}(k-2)k^{-2s}\e^{q(k-2)}>0$
for $q<0$ close to $0$ and $\sum_{k=1}^{\infty}(k-2)k^{-2s}\e^{q(k-2)}\to-\infty$
for $q\to-\infty$, there exists a unique number $q\left(s\right)<0$
with $\int\psi\d_{\mu_{s\varphi+q(s)\psi}}=0$. Hence, for $s<s_{0}$
we have that $\inf_{q\leq0}\mathcal{P}(s\varphi+q\psi)=\mathcal{P}(s\varphi+q(s)\psi)$.

For $s>s_{0}$ we have 
\[
\int\psi\d\mu_{s\varphi}<0,
\]
and by the variational principle
\[
\inf_{q\leq0}\mathcal{P}(s\varphi+q\psi)=\mathcal{P}(s\varphi).
\]
Combining with our variational formula \prettyref{eq:(3)}, we have
thus shown 
\[
\mathcal{P}(s\varphi,\N)=\inf_{q\leq0}\mathcal{P}(s\varphi+q\psi)=\begin{cases}
\mathcal{P}(s\varphi), & s\geq s_{0},\\
\mathcal{P}(s\varphi+q(s)\psi), & s<s_{0}.
\end{cases}
\]
This function is analytic for $s\neq s_{0}$ by the implicit function
theorem. For the first derivative we have
\[
\frac{\partial}{\partial s}\mathcal{P}(s\varphi,\N)=\begin{cases}
\frac{\partial}{\partial s}\mathcal{P}(s\varphi)=\frac{\sum_{k=1}^{\infty}-2\log(k)k^{-2s}}{\sum_{k=1}^{\infty}k^{-2s}}, & s>s_{0},\\
\frac{\partial}{\partial s}\mathcal{P}(s\varphi+q(s)\psi)=\frac{\sum_{k=1}^{\infty}-2\log(k)k^{-2s}\e^{q\left(s\right)(k-2)}}{\sum_{k=1}^{\infty}k^{-2s}\e^{q\left(s\right)(k-2)}}, & s<s_{0}.
\end{cases}
\]

Next we show $q(s)\rightarrow0$ for $s\nearrow s_{0}$. First observe
that $q=0$ is the unique zero of $q\mapsto\sum_{k=1}^{\infty}(k-2)k^{-2s_{0}}\e^{q(k-2)}$.
Since for $q\ge1.1$ we have 
\[
\sum_{k=1}^{\infty}(k-2)k^{-2s}\e^{q(k-2)}\le-\e^{-q}+\sum_{k=1}^{\infty}k(k+2)^{-2.2},
\]
it follows that $s\mapsto q(s)$ is bounded on $[1.1,s_{0})$. Since
\[
(1,\infty)\times(-\infty,0]\ni(s,q)\mapsto\int\psi\d\mu_{s\varphi+q\psi}=\e^{-\mathcal{P}(s\varphi+q\psi)}\sum_{k=1}^{\infty}(k-2)k^{-2s}\e^{q(k-2)}
\]
is continuous, we conclude that the only accumulation point of $q(s)$,
as $s\nearrow s_{0}$, is $0$. we have thus shown that $q(s)\rightarrow0$
for $s\nearrow s_{0}$. Since $q(s)\rightarrow0$ as $s\nearrow s_{0}$,
by the mean value theorem we find that $s\mapsto\mathcal{P}(s\varphi,\N)$
is also differentiable at $s=s_{0}$ and the derivative is continuous.
The second derivative of $s\mapsto\mathcal{P}(s\varphi,\N)$ exists
on $\R\setminus\left\{ s_{0}\right\} $. Indeed, for $s>s_{0}$, 
\[
\frac{\partial^{2}}{\partial s^{2}}\mathcal{P}(s\varphi,\N)=\frac{\partial^{2}}{\partial s^{2}}\mathcal{P}(s\varphi)=\cov_{s\varphi}\left(\varphi,\varphi\right)=\frac{1}{\zeta\left(2s\right)}\sum_{k=1}^{\infty}(k-2)^{2}k^{-2s}=\cov_{s\varphi}\left(\varphi,\varphi\right),
\]
and for $s<s_{0}$, 
\[
\frac{\partial^{2}}{\partial s^{2}}\mathcal{P}(s\varphi,\N)=\frac{\partial^{2}}{\partial s^{2}}\mathcal{P}(s\varphi+q(s)\psi)=\cov_{s\varphi+q(s)\psi}\left(\varphi,\varphi\right)-\frac{\left(\cov_{s\varphi+q(s)\psi}\left(\varphi,\psi\right)\right)^{2}}{\cov_{s\varphi+q(s)\psi}\left(\psi,\psi\right)}.
\]
Since $s\mapsto\mathcal{P}(s\varphi)$ is analytic in a neighborhood
of $s_{0}$, 
\[
\lim_{t\searrow s_{0}}\frac{\partial^{2}}{\partial s^{2}}\mathcal{P}(s\varphi,\N)_{s=t}=\frac{\partial^{2}}{\partial s^{2}}\mathcal{P}(s\varphi)_{s=s_{0}}=\cov_{s_{0}\varphi}\left(\varphi,\varphi\right).
\]
By the mean value theorem, we have 
\[
\lim_{t\nearrow s_{0}}\frac{\partial^{2}}{\partial s^{2}}\mathcal{P}(s\varphi,\N)_{s=t}=\lim_{s\nearrow s_{0}}\cov_{s\varphi+q(s)\psi}\left(\varphi,\varphi\right)-\frac{\left(\cov_{s\varphi+q(s)\psi}\left(\varphi,\psi\right)\right)^{2}}{\cov_{s\varphi+q(s)\psi}\left(\psi,\psi\right)}.
\]
By the Lebesgue dominated convergence theorem, we obtain as $s\nearrow s_{0}$,
\begin{align*}
 & \cov_{s\varphi+q(s)\psi}\left(\varphi,\varphi\right)\\
 & =\int\varphi^{2}\d\mu_{s\varphi+q(s)\psi}-\left(\int\varphi\d\mu_{s\varphi+q(s)\psi}\right)^{2}\\
 & =\e^{-\mathcal{P}(s\phi+q(s)\psi)}\sum_{k=1}^{\infty}\left(-2\log k\right)^{2}k^{-2s}\e^{q(s)(k-2)}-\left(\e^{-\mathcal{P}\left(s\phi+q(s)\psi\right)}\sum_{k=1}^{\infty}\left(-2\log k\right)k^{-2s}\e^{q(s)(k-2)}\right)^{2}\\
 & \rightarrow\frac{1}{\zeta\left(2s_{0}\right)}\sum_{k=1}^{\infty}\left(-2\log k\right)^{2}k^{-2s_{0}}-\left(\frac{1}{\zeta\left(2s_{0}\right)}\sum_{k=1}^{\infty}\left(-2\log k\right)k^{-2s_{0}}\right)^{2}=\cov_{s_{0}\varphi}\left(\varphi,\varphi\right).
\end{align*}
Similarly, as $s\nearrow s_{0}$ we have 
\begin{align*}
\cov_{s\varphi+q(s)\psi}\left(\varphi,\psi\right) & =\int\varphi\psi\d\mu_{s\varphi+q(s)\psi}=\e^{-\mathcal{P}(s\varphi+q(s)\psi)}\sum_{k=1}^{\infty}\left(-2\log k\right)(k-2)k^{-2s}e^{q(s)(k-2)}\\
 & \rightarrow\frac{1}{\zeta\left(2s_{0}\right)}\sum_{k=1}^{\infty}\left(-2\log k\right)(k-2)k^{-2s_{0}}=\cov_{s_{0}\varphi}\left(\varphi,\psi\right)
\end{align*}
and
\begin{align*}
\cov_{s\varphi+q(s)\psi}\left(\psi,\psi\right) & =\int\psi^{2}\d\mu_{s\varphi+q(s)\psi}=\e^{-\mathcal{P}\left(s\phi+q(s)\psi\right)}\sum_{k=1}^{\infty}(k-2)^{2}k^{-2s}e^{q(s)(k-2)}\\
 & \rightarrow\frac{1}{\zeta\left(2s_{0}\right)}\sum_{k=1}^{\infty}(k-2)^{2}k^{-2s_{0}}=\cov_{s_{0}\varphi}\left(\varphi,\varphi\right).
\end{align*}
Hence, for the second derivative at $s_{0}$ we have a phase transition
due to \eqref{eq:cov-neg}. 

\subsection{Spectral radius of Hessenberg matrices}

Formula \prettyref{eq:(3)} can be used to compute the spectral radius
$\rho_{\mathfrak{H}}$ of the infinite-dimensional \emph{Hessenberg
matrix for a given positive sequence $\left(a_{i}\right)_{i\ge0}\in\ell^{1}$,
$M\ge0$ and with $a_{j}=0$ for $j<0$, }
\[
\mathfrak{H}\coloneqq\left(\begin{array}{cccccc}
a_{0-M} & a_{1-M} & a_{2-M} & a_{3-M} & a_{4-M} & \cdots\\
a_{0} & a_{1} & a_{2} & a_{3} & a_{4} & \cdots\\
0 & a_{0} & a_{1} & a_{2} & a_{3} & \cdots\\
\vdots & 0 & a_{0} & a_{1} & a_{2} & \cdots\\
\vdots & \vdots & 0 & a_{0} & a_{1} & \ddots\\
\vdots & \vdots & \vdots & \ddots & \ddots & \ddots
\end{array}\right).
\]
The \emph{spectral radius of} $\mathfrak{H}$ \cite{MR1743100} is
for each $i\in\N$, and by irreducibility independently of $i$, given
by
\[
\rho_{\mathfrak{H}}=\limsup_{n\rightarrow\infty}\left(\left(\mathfrak{H}^{n}\right)_{i,i}\right)^{1/n}.
\]
 Interpreting the $a_{i}$'s as the inverse of the slopes of full
linear branches of an interval map, $\varphi(x):=\log\left(a_{i}\right)$,
$\psi\left(x\right):=i-1$ as the step length function for $x\in J_{i}$,
and $\psi_{1}:=\psi+M+1$ as the reflexion rule, we find 
\[
\mathcal{P}(\varphi+q\psi)=\log\left(\sum_{k\in\N}a_{k}\e^{q(k-1)}\right)\;\text{ and \;}\log\rho_{\mathfrak{H}}=\mathcal{P}\left(\varphi,\N\right).
\]
Note that $q\mapsto\mathcal{P}(\varphi+q\psi)$ is a convex analytic
function on $(-\infty,0)$ and $\mathcal{P}(\varphi)<\infty$. Hence,
by a direct application of \prettyref{eq:(3)}, we get
\[
\log\rho_{\mathfrak{H}}=\inf_{q\le0}\mathcal{P}(\varphi+q\psi)=\begin{cases}
\log\left({\displaystyle \sum_{k\in\N}a_{k}}\right), & \forall q\leq0:{\displaystyle \sum_{k\in\N}\left(k-1\right)a_{k}\e^{qk}<0},\\
\log\left({\displaystyle \sum_{k\in\N}a_{k}\e^{q_{0}(k-1)}}\right), & \exists!q_{0}<0:{\displaystyle \sum_{k\in\N}\left(k-1\right)a_{k}\e^{q_{0}k}=0.}
\end{cases}
\]

If we are in the special situation $a_{k}\coloneqq\lambda^{k}$, $k\in\N$
with $\lambda\in\left(0,1\right)$, we write $\mathfrak{H}_{\lambda}$
for the corresponding Hessenberg matrix, and the calculation is simplified
to 
\[
\mathcal{P}(\varphi+q\psi)=\log\left(\sum_{k\in\N}\lambda^{k}\e^{q(k-1)}\right).
\]
For $\lambda\geq1/2$, $q\mapsto\mathcal{P}(\varphi+q\psi)$ is monotone
decreasing on $\R_{\leq0}$ and for $\lambda\leq1/2$ has a unique
minimum in $q_{0}=-\log\left(2\lambda\right)$. Therefore, according
to \prettyref{eq:(3)}, for $\rho_{\lambda}\coloneqq\rho_{\mathfrak{H}_{\lambda}}$
we immediately obtain
\[
\log\rho_{\lambda}=\begin{cases}
-\log(1-\lambda), & \lambda\ge1/2,\\
\log(4\lambda), & \lambda\le1/2.
\end{cases}
\]

The spectral radius $\rho_{\lambda}$ has also been computed in \cite[Theorems 1 and 2]{MR2959300}
by two different methods: the first is by solving recursive formul\ae~for
the eigenvectors of $\mathfrak{H}_{\lambda}$ giving rise to conformal
measures, the second is by computation of the maximal eigenvalue of
the $k\times k$ sub-matrices of $\mathfrak{H}_{\lambda}$ and letting
$k$ tend to infinity. In contrast to this, our approach is based
on drift arguments for $\Z$-extensions and a careful analysis of
the effect of the reflective boundary (see \prettyref{thm:Onesided_fibre-pressure_vs_base-pressure}).

\section{Thermodynamic formalism for $\Psi$-extensions \label{subsec:Thermodynamic-formalism}}

\subsection{\label{subsec:symbolic_ifs}Symbolic dynamics for expanding Markov
maps}

Let us return to the general setting as introduced in \prettyref{subsec:Expanding-interval-maps}.
Let $f:\Delta\to\R$ be an expanding $C^{2}$-Markov interval map
with $\Delta=\bigcup_{i\in I}J_{i}$. 

For each $i\in I$ there exists $k(i)\in\Z$ and $\ell(i)\in\Z$ such
that $f_{i}(\overline{J_{i}})=[k(i),k(i)+1]$ and $\overline{J_{i}}\subset[\ell(i),\ell(i)+1]$.
The inverse branches $f_{i}{}^{-1}:[k(i),k(i)+1]\rightarrow\overline{J_{i}}$,
$i\in I$, give rise to a graph directed Markov system $(f_{i}^{-1})_{i\in I}$
in the sense of \cite[Section 4.2]{MR2003772} with possibly infinitely
many vertices. Note that the number of vertices is required to be
finite in \cite{MR2003772}.

Define the Markov shift 
\[
\Sigma:=\left\{ \omega=(\omega_{0},\omega_{1},\dots)\in I^{\N}:J_{\omega_{k+1}}\subset f\left(J_{\omega_{k}}\right),k\ge0\right\} \subset I{}^{\N}.
\]
For $\omega=(\omega_{0},\omega_{1},\dots)\in I^{\N}$ and $n\in\N$
we define $\omega|_{n}:=(\omega_{0},\dots,\omega_{n-1})$. Define
the set of admissible $n$-words $\Sigma^{n}:=\left\{ \omega|_{n}:\omega\in\Sigma\right\} $.
For $\omega=\left(\omega_{0},\ldots,\omega_{n-1}\right)\in\Sigma^{n}$,
and $n\in\N$ let $k\in\Z$ with $f\left(J_{\omega_{n-1}}\right)=\left(k,k+1\right)$.
Then $J_{\omega}\coloneqq f_{\omega_{0}}^{-1}\circ\dots\circ f_{\omega_{n-1}}^{-1}(k,k+1)$
defines a non-empty subinterval of $J_{\omega_{0}}$. The \emph{coding
map} $\pi=\pi_{f}:\Sigma\rightarrow\RepMU(f)$ of $f$ is given by
\[
\bigcap_{n\in\N}\overline{J}_{\omega|_{n}}=\left\{ \pi(\omega)\right\} .
\]

We have $\pi(\Sigma)=\RepMU(f)$ and $\pi^{-1}(x)$ is a singleton
for $x\in\Rep(f)\subset\RepMU\left(f\right)$. Since $\RepMU\left(f\right)\setminus\Rep\left(f\right)$
is countable, it follows that $\pi$ defines up to a countable set
a topological conjugacy between $f:\Rep(f)\rightarrow\Rep(f)$ and
$\sigma:\Sigma\rightarrow\Sigma$. Since $f$ is topological transitive,
$\sigma$ is topological transitive Markov shift.

Let $g:\Rep(f)\rightarrow\R$ be H\"older continuous. It then follows
that $g\circ\pi:\pi^{-1}\left(\Rep\left(f\right)\right)\rightarrow\R$
is H\"older continuous with respect to the shift metric, that is,
there exist two constants $L\ge1$ and $\alpha>0$ such that for all
$\omega,\tau\in\pi^{-1}\left(\Rep(f)\right)$, 
\[
\left|g(\pi\omega)-g(\pi\tau)\right|\le L\e^{-\alpha\left|\omega\wedge\tau\right|}
\]
where $\left|\omega\wedge\tau\right|$ denotes the length of the longest
common initial block of $\omega$ and $\tau$ \cite[Section 4.2]{MR2003772}.
It is well known that $g\circ\pi$ has a unique H\"older continuous
extension to $\Sigma$ with the same $\alpha$ and $L$, which we
will also denote by $g\circ\pi:\Sigma\rightarrow\R$.  Note that
this extension is for $\omega\in\Sigma$ with $\omega_{0}=i$ given
by $g\circ\pi(\omega)=g_{i}(\pi(\omega))$. For $\omega\in\Sigma^{n}$
we define the cylinder set $\left[\omega\right]:=\left\{ x\in\Sigma:x|_{n}=\omega\right\} $.

The following\emph{ bounded distortion property} holds:
\begin{align*}
D_{g} & :=\sup_{n\in\N}\sup_{\omega\in\Sigma^{n}}\sup_{x,y\in J_{\omega}}S_{n}g(x)-S_{n}g(y)\\
 & =\sup_{n\in\N}\sup_{\omega\in\Sigma^{n}}\sup_{x,y\in[\omega]}S_{n}^{\sigma}\left(g\circ\pi\right)(x)-S_{n}^{\sigma}\left(g\circ\pi\right)(y)<+\infty.
\end{align*}
The geometric potential is given by $\varphi:=-\log\left|f'\right|:\Delta\rightarrow\R_{<0}$.
By our assumption these functions are Lipschitz and by the Rényi condition
the Lipschitz constant can be chosen uniformly in $i\in I$.

Let $\Per_{n}\left(\sigma\right)\coloneqq\left\{ x\in\Sigma:\sigma^{n}x=x\right\} $
denote the set of $n$-periodic points of $\sigma$. For $g\circ\pi:\Sigma\rightarrow\R$
H\"older continuous and for each $k\in\N$ and $\nu\in\Sigma^{k}$,
the Gurevich pressure in the symbolic setting is given by (cf\@.
\cite{MR1738951})
\begin{equation}
\mathcal{P}\left(g\circ\pi,\sigma\right):=\limsup_{n\rightarrow\infty}\frac{1}{n}\log\sum_{\omega\in\Per_{n}\left(\sigma\right)\cap\left[\nu\right]}\exp S_{n}^{\sigma}\left(g\circ\pi\right)(\omega).\label{eq:GurevichPressureSymbolic}
\end{equation}
Since $\Sigma$ is irreducible, this definition is independent of
$i\in I$.

The following lemma guarantees that we can switch between the symbolic
and the geometric setting in the definition of pressure.
\begin{lem}
\label{lem:L_impies_in_Rep(f)} For $\nu\in\Sigma^{k}$, $k\geq1$,
and $L=J_{\nu}$ a basic interval of order $k$, assume $\overline{L}\cap\Z=\emptyset$.
Then $\pi:\Per_{n}\left(f\right)\cap\left[\nu\right]\to\Per_{n}\left(f\right)\cap L$
is a bijection for all $n\geq1$.
\end{lem}

\begin{proof}
Let $\omega\in\Per_{n}\left(\sigma\right)\cap\left[\nu\right]$, i.\,e\@.
$\sigma^{n}\omega=\omega$ and $\omega|_{k}=\nu$. Hence, $x=\pi\left(\omega\right)\in\overline{\Rep\left(f\right)\cap L}$
and $f_{\omega_{n-1}}\circ\dots\circ f_{\omega_{0}}(x)=x$. Let $\ell\in\N$
such that $\overline{L}\subset(\ell,\ell+1)$. Since $f_{\omega_{n-1}}\circ\dots\circ f_{\omega_{0}}:\overline{J_{\omega|_{n}}}\rightarrow\left[\ell,\ell+1\right]$
is a homeomorphism and $x\in\overline{L}\subset(\ell,\ell+1)$, we
have $x\in J_{\omega|_{n}}\cap L$. Hence, $x\in\Per_{n}\left(f\right)\cap L$.
Since $J_{\omega|_{n}}$ are pairwise disjoint for $\omega\in\Per_{n}\left(\sigma\right)\cap\left[\nu\right]$,
the map $\pi$ is one-to-one on $\Per_{n}\left(\sigma\right)\cap\left[\nu\right]$.
Conversely, let $x\in\Per_{n}\left(f\right)\cap L$. Then for $\omega\coloneqq\pi^{-1}\left(x\right)\in\Sigma$
we have $\omega|_{k}=\nu$ and $\sigma^{n}\omega=\omega$.
\end{proof}
\begin{prop}
\label{prop:gurevich_wrt_f} Let $g:\Rep(f)\rightarrow\R$ be H\"older
continuous. Then, for any admissible word $\nu\in\Sigma^{n}$,
\[
\mathcal{P}\left(g\circ\pi,f\right)=\limsup_{n}\frac{1}{n}\log\sum_{x\in\Per_{n}\left(f\right)\cap J_{\nu}}\exp S_{n}g\left(x\right).
\]
\end{prop}

\begin{proof}
Since $f$ is conjugate to $\sigma|_{\pi^{-1}\Rep(f)}$, we infer
the upper bound 
\[
\limsup_{n}\frac{1}{n}\log\sum_{x\in\Per_{n}\left(f\right)\cap J_{\nu}}\exp S_{n}^{f}g\left(x\right)\leq\limsup_{n\rightarrow\infty}\frac{1}{n}\log\sum_{\omega\in\Per_{n}\left(\sigma\right)\cap\left[\nu\right]}\exp S_{n}^{\sigma}\left(g\circ\pi\right)(\omega)
\]
and if $J_{\nu}\cap\Z=\emptyset$ then by \prettyref{lem:L_impies_in_Rep(f)}
equality holds. If otherwise $J_{\nu}\cap\Z\neq\emptyset$, then we
consider an appropriate subword $\nu'$ such that $J_{\nu'}\cap\Z=\emptyset$
and derive a lower bound which coincides with $\mathcal{P}(g\circ\pi,\sigma)$
since its definition is independent from the chosen start cylinder.
\end{proof}
In the following we consider the special situation of a Markov $\FI_{0}:\Delta_{0}\to\left(0,1\right)$
without any extension, i.\,e\@. $H=0$ and show that the Gurevich
pressure and the classical topological pressure with respect to an
infinite alphabet as defined e.\,g\@. in \cite{MR2003772} coincide.
\begin{cor}
\label{cor:GurevichMU-Pressure} Let $\FI_{0}:\Delta_{0}\to\left(0,1\right)$
be an expanding $C^{2}$-Markov interval map. Then
\[
\mathcal{P}\left(g\right)=\lim_{n}\frac{1}{n}\log\sum_{\omega\in\Sigma^{n}}\e^{S_{\omega}g}
\]
with $S_{\omega}g\coloneqq\sup_{x\in\left[\omega\right]}S_{n}^{\sigma}g\circ\pi\left(x\right)$.
\end{cor}

\begin{proof}
Since the corresponding shift $\Sigma$ is full it is well know that
the Gurevich pressure and the classical topological pressure of $g\circ\pi$
with respect to the shift dynamical system coincide. The equality
then follows from \prettyref{prop:gurevich_wrt_f}.
\end{proof}
\begin{rem}
Finally, we remark that for any $\Psi$-extension $\FI_{\H}:\Delta_{\H}\to\left(0,1\right)$
with $0\ni\H\subset\Z$ and Hölder continuous potential $g:\Delta_{0}\to\R$
and $g_{\H}$ its periodic continuous extension, we have 
\begin{equation}
\sum_{k=0}^{n-1}g\circ\FI_{0}^{k}\left(x\right)=\sum_{k=0}^{n-1}g_{\H}\circ\FI_{\H}^{k}\left(x\right).\label{eq:Skew_ergodic_sum}
\end{equation}
\end{rem}

For $K\subset I$ we write $f_{K}$ for the restriction of $f$ to
$\Delta_{K}\coloneqq\biguplus_{i\in K}J_{i}$, that is $f_{K}:\Delta_{K}\rightarrow\R$
defines an expanding $C^{1+\epsilon}$-Markov interval map and the
restricted potential function is given by $g_{K}\coloneqq g|_{\Rep\left(f_{K}\right)}$.

\begin{lem}
\label{lem:exhaustio_principle}The following exhaustion principle
holds
\[
\mathcal{P}(g,f)=\sup_{K\subset I,\,\#K<\infty}\mathcal{P}(g_{K},f_{K}).
\]
\end{lem}

\begin{proof}
Combine \prettyref{prop:gurevich_wrt_f} with Sarig's compact approximation
property.
\end{proof}

\subsection{\label{subsec:Gibbs-measures-and}Gibbs measures}

Let $\FI_{0}:\Delta_{0}\to\left(0,1\right)$ be an expanding $C^{2}$-Markov
interval map and $g:\Rep_{0}\rightarrow\R$ be H\"older continuous.
Then by \prettyref{cor:GurevichMU-Pressure} we have that the our
Gurevich pressure coincides with the classical topological pressure
with respect to an at most countable alphabet. Since $g\circ\pi:I^{\N}\rightarrow\R$
is also H\"older continuous and if $\mathcal{P}\left(g\right)<\infty$,
then by \cite{MR2003772} there exists a unique $\sigma$-invariant
Borel probability measure $\mu_{g}$ on $I^{\N}$ -- called the \emph{Gibbs
measure} (for $g$) -- such that for some $C\geq1$ we have for all
$\omega\in I^{\ast}$
\begin{equation}
C^{-1}\leq\frac{\mu_{g}\left(\left[\omega\right]\right)}{\exp\left(S_{\omega}g-\left|\omega\right|\mathcal{P}(g)\right)}\leq C.\label{eq:Gibbs}
\end{equation}
Note that $\mu_{g}$ is the unique equilibrium state for $g$ with
respect to $\FI_{0}$. Since $\mu_{g}$ has no atoms, we can consider
$\mu_{g}$ also as a probability measure on $\Rep_{0}$.

We say that $\FI_{0}$ is \emph{strongly regular} \cite{MR2003772}
iff there exists $s>0$ such that $0<\mathcal{P}\left(s\varphi\right)<\infty$.
If $\FI_{0}$ is strongly regular, then $\mathcal{P}(\delta_{0}\varphi)=0$
where $\delta_{0}=\dim_{H}\left(\Rep_{0}\right)$ and there exists
a unique Gibbs measure $\mu_{\delta_{0}\varphi}$ on $\Rep_{0}$,
resp. on $I^{\N}$.

\subsection{Non-reflective case}

In this section, let $\FI_{0}:\Delta_{0}\to\left(0,1\right)$ be an
expanding $C^{2}$-Markov interval map, $g:\Rep_{0}\rightarrow\R$
is Hölder continuous and $\psi:\Rep_{0}\to\Z$ is constant on basic
intervals of level one, that is we consider the particular choice
$H=\Z$ and $\Psi_{\Z}:\Delta_{0}+\Z\to\Z$, $x\mapsto\psi\left(x-\left\lfloor x\right\rfloor \right)+\left\lfloor x\right\rfloor $.
We always assume that $\Psi_{\Z}$ is surjective, which is equivalent
to the transitivity of $\FI_{\Z}.$

The following theorem and lemma can be shown with the same arguments
as in \cite{MR4373996} combined with \prettyref{prop:gurevich_wrt_f}
and the exhaustion argument provided in \prettyref{lem:exhaustio_principle}.
\begin{thm}[Variational Formula for $\Z$-extensions]
\label{thm:fibre_pressure_vs_base_pressure}Suppose $g:\Rep\left(\FI_{0}\right)\to\R_{<0}$
Hölder continuous. Then we have
\begin{enumerate}
\item If $0\in\left(\underline{\psi},\overline{\psi}\right)$, then there
exists a unique $t\in\R$ such that 
\[
\mathcal{P}\left(g,\Z\right)=\mathcal{P}\left(t\psi+g\right)=\min{}_{q\in\R}\mathcal{P}\left(q\psi+g\right).
\]
\item In general, we have
\[
\mathcal{P}\left(g,\Z\right)=\inf{}_{q\in\R}\mathcal{P}\left(q\psi+g\right)
\]
 and $\mathcal{P}\left(g,\Z\right)=-\infty$ if $0\notin\left[\underline{\psi},\overline{\psi}\right]$.
If $0\in\left\{ \underline{\psi},\overline{\psi}\right\} $ then $\mathcal{P}\left(g,\Z\right)=\mathcal{P}\left(g_{I_{0}}\right)$
where $I_{0}=\left\{ i\in I:\psi(i)=0\right\} $.
\end{enumerate}
\end{thm}

\begin{lem}
\label{lem:infimum_exchange}
\[
\inf\left\{ s\ge0:\inf_{q}\mathcal{P}\left(q\psi+s\varphi\right)\le0\right\} =\inf_{q}\inf\left\{ s\ge0:\mathcal{P}\left(s\varphi+q\psi\right)\le0\right\} 
\]
and 
\[
\inf\left\{ s\ge0:\inf_{q\le0}\mathcal{P}\left(q\psi+s\varphi\right)\le0\right\} =\inf_{q\le0}\inf\left\{ s\ge0:\mathcal{P}\left(s\varphi+q\psi\right)\le0\right\} .
\]
\end{lem}

\begin{lem}
\label{lem:inf_sq_criterion} With $s(q):=\inf\left\{ s\ge0:\mathcal{P}\left(s\varphi+q\psi\right)\le0\right\} $
for $q\in\R$, we have
\[
\inf_{q\le0}s(q)=\begin{cases}
\delta_{\Z}, & \text{if }\alpha_{\max}\ge0,\\
\delta_{0}, & \text{if }\text{\ensuremath{\alpha_{\max}\le0}}.
\end{cases}
\]
\end{lem}

\begin{proof}
First assume $\alpha_{\max}=\infty$. If $\mathcal{P}\left(\delta_{0}\varphi+q\psi\right)<\infty$
for some $q>0$, then $\alpha_{\max}<\infty$ by \cite[Proposition 2.6.13]{MR2003772}.
Hence, in this case, $\mathcal{P}\left(\delta_{0}\varphi+q\psi\right)=\infty$
for every $q>0$ and consequently $\inf_{q}s(q)=\inf_{q\le0}s(q)$.
Combining this with \prettyref{lem:infimum_exchange} gives $\inf_{q\le0}s(q)=\delta_{\Z}$.

Let $\alpha_{\max}<\infty$ and suppose $s(q)<\infty$ for $q\in\R$.
That is, there is $s>0$ with $\mathcal{P}\left(s\varphi+q\psi\right)\leq0$.
Then by an application of the variational principle (see \cite{MR1853808},
using $\int s\varphi+q\psi\d\mu_{\delta_{0}\varphi}>-\infty$ and
$\delta_{0}=h\left(\mu_{\delta_{0}\varphi}\right)/\int-\varphi\d\mu_{\delta_{0}\varphi}$)
we find 
\[
s\geq q\frac{\alpha_{\max}}{-\int\varphi\d\mu_{\delta_{0}\varphi}}+\delta_{0}.
\]
Taking infimum over $\left\{ s>0:\mathcal{P}\left(s\varphi+q\psi\right)\leq0\right\} $
gives $s(q)\geq q\alpha_{\max}/-\int\varphi\d\mu_{\delta_{0}\varphi}+\delta_{0}$
and $s(0)=\delta_{0}$. If $s(q)=\infty$ the inequality holds trivially.
Hence 
\begin{equation}
s(q)\geq q\frac{\alpha_{\max}}{-\int\varphi\d\mu_{\delta_{0}\varphi}}+\delta_{0}\quad\text{for all }q\in\R.\label{eq:sub_differential}
\end{equation}
It follows that $q\mapsto s\left(q\right)$ is minimised on the non-positive
if $\alpha_{\max}\ge0$. Hence, $\inf_{q\le0}s(q)=\inf_{q}s(q)=\delta_{\Z}$,
where the last equality follows from \prettyref{thm:fibre_pressure_vs_base_pressure}
and \prettyref{lem:infimum_exchange}.

\begin{figure}
\begin{tikzpicture}[scale=.9,line cap=round,line join=round,>=triangle 45,x=2cm,y=2cm] \begin{axis}[ x=6cm,y=6cm, axis lines=middle, axis line style = {-latex, thick}, ymajorgrids=false, xmajorgrids=false, xmin=-.7, xmax=0.4 , ymin=-0.15, ymax=0.8 , xtick={-0.352,-0.5 }, xticklabels={$q_0$,-0.5 }, ytick={ 0,0.204,0.44,0.7},yticklabels={0,$\delta_{\Psi_\Z}=\delta_{\Psi_\N}$,$\delta_0$,0.7}, xlabel={\normalsize $q$},            
ylabel={\normalsize $s(q)$}, major tick length=2mm,                 
every y tick label/.style={anchor=near yticklabel opposite,xshift=\pgfkeysvalueof{/pgfplots/major tick length}},] \clip(-1.2524882467730518,-0.3838949819586078) 
rectangle (0.6,1.2); 
\draw[line width=1.1pt, samples=40,domain=-1:0.4] plot(\x,{log10(exp(8*\x + 3) + exp(-3*\x)) / 3}); 
\draw [line width=0.5pt,dash pattern=on 2pt off 2pt,domain=-1.2:-0.0] plot(\x,{(--0.204-0*\x)/1});
\draw [line width=0.5pt,dash pattern=on 2pt off 2pt] (-0.352,-0) -- (-0.352,0.5); 
\end{axis} 
\end{tikzpicture}\hspace*{4mm}\begin{tikzpicture}[scale=.9,line cap=round,line join=round,>=triangle 45,x=2cm,y=2cm] \begin{axis}[ x=6cm,y=6cm, axis lines=middle, axis line style = {-latex, thick}, ymajorgrids=false, xmajorgrids=false, xmin=-.35, xmax=0.65 , ymin=-0.15, ymax=0.8 , xtick={0.32 ,0.5}, xticklabels={$q_0$ ,0.5}, ytick={ 0,0.198,0.485,0.7},yticklabels={0,$\delta_{\Psi_\Z}$,$\delta_0=\delta_{\Psi_\N}$,0.7}, xlabel={\normalsize $q$},            
ylabel={\normalsize $s(q)$}, major tick length=2mm,                 
] \clip(-1.2524882467730518,-0.3838949819586078) 
rectangle (0.6,1.2); 
\draw[line width=1.1pt, samples=40,domain=-0.3:0.5] plot(\x,{log10(5*exp(8*\x - 5) + 3*exp(-3*\x)) / 1}); 
\draw [line width=0.5pt,dash pattern=on 2pt off 2pt,domain=0:1.2] plot(\x,{(0.198-0*\x)/1});
\draw [line width=0.5pt,dash pattern=on 2pt off 2pt] (0.32,-0) -- (0.32,0.5); 
\end{axis} 
\end{tikzpicture}\caption{\label{fig:The-grph-of-s(q)}The graph of $q\protect\mapsto s(q)$
with the values $\delta_{0}$, $\delta_{\Z}$ and $\delta_{\N}$ for
$\alpha_{\max}>0$ (left) and for $\alpha_{\max}<0$ (right).}
\end{figure}
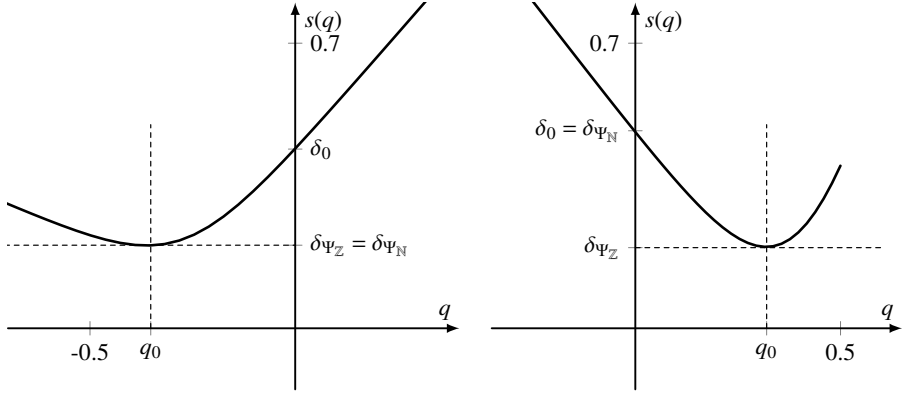
Assume $\alpha_{\max}\le0$. Then, using \prettyref{eq:sub_differential}
for $q$ negative, we deduce $\inf_{q\le0}s(q)=s(0)=\delta_{0}$.
Two typical graphs for $q\mapsto s(q)$ are illustrated in \prettyref{fig:The-grph-of-s(q)}.
\end{proof}
\begin{lem}
\label{lem:convergence_of_gibbs_measures}For $s\in\R$ such that
$\mathcal{P}\left(s\varphi\right)<\infty$ we have
\[
\liminf_{q\nearrow0}\int\psi\d\mu_{s\varphi+q\psi}\geq\int\psi\d\mu_{s\varphi}.
\]
\end{lem}

\begin{proof}
Since $\psi\geq-1$ we have for $\eta<0$, uniformly in $q\in\left[\eta,0\right]$,
\[
\int-\varphi\d\mu_{s\varphi+q\psi}\asymp\sum_{i}\sup\left(-\varphi|_{J_{i}}\right)\left(\exp\sup\varphi|_{J_{i}}\right)^{s}\left(\exp\sup\psi|_{J_{i}}\right)^{q}\exp\mathcal{P}\left(s\varphi+q\psi\right)\ll1.
\]
Therefore, by \cite[Lemma 3.2]{MR2343687} we have that $\left(\mu_{s\varphi+q\psi}:q\in\left[\eta,0\right]\right)$,
is a tight family of probability measures. The Gibbs property \prettyref{eq:Gibbs}
and the continuity of the pressure shows that we have weak-{*} convergence
of $\mu_{s\varphi+q\psi}$ to $\mu_{s\varphi}$ as $q\nearrow0$.
The lower semi-continuity of $\mu\mapsto\int\psi\d\mu$ then proves
the lemma.
\end{proof}
\begin{proof}[Proof of \prettyref{thm:DimGap_non-reflective}]
 We will make use of the fact that $\mathcal{P}\left(\delta_{0}\varphi+q\psi\right)<\infty$
if and only if $\sum_{x\in\Sigma:\sigma x=x}\exp\left(\left(\delta_{0}\varphi\left(x\right)+q\psi\left(x\right)\right)\right)<\infty$
(see \cite{MR2003772}). By \prettyref{thm:fibre_pressure_vs_base_pressure}
(2) we have a dimension gap if and only if 
\[
\inf\left\{ s>0:\inf_{q}\mathcal{P}\left(s\varphi+q\psi\right)\leq0\right\} <\delta_{0}.
\]
Suppose $\mathcal{P}\left(s\varphi+q'\psi\right)<\infty$ for some
$q'\in\R$, then, using $-1\leq\psi$, for $q\leq q'$ we have
\begin{align*}
\frac{1}{n}\log\sum_{x=\FI_{0}^{n}x}\e^{S_{n}s\varphi(x)+q\psi(x)} & \leq\frac{1}{n}\log\sum_{\omega\in\Sigma^{n}}\e^{S_{\omega}s\varphi+q'\psi+\left(q-q'\right)\psi}\\
 & \leq\frac{1}{n}\log\sum_{\omega\in\Sigma^{n}}\e^{S_{\omega}\left(s\varphi+q'\psi-\left(q-q'\right)\right)}\\
 & \leq q'-q+\frac{1}{n}\log\sum_{\omega\in\Sigma^{n}}\e^{S_{\omega}s\varphi+q'\psi}
\end{align*}
and consequently $\mathcal{P}\left(s\varphi+q\psi\right)<\infty$
for all $q\leq q'$. Using the regularity assumption it follows that
$p:\left(s,q\right)\mapsto\mathcal{P}\left(s\varphi+q\psi\right)$
is finite and strictly decreasing in the first coordinate on $\left(\delta_{0}-\epsilon,\infty\right)\times\left(-\infty,0\right]$
for some $\epsilon>0$. Assume $\alpha_{\max}>0$. We show that on
$\left(\delta_{0}-\epsilon,\infty\right)\times\left(-\epsilon,0\right]$,
for a possibly smaller $\epsilon>0$, the function $p$ is also strictly
increasing in the second coordinate. Indeed, \prettyref{lem:convergence_of_gibbs_measures}
implies $\liminf_{q\nearrow0}\int\psi\d\mu_{\delta_{0}\varphi+q\psi}\geq\int\psi\d\mu_{\delta_{0}\varphi}=\alpha_{\max}>0$
forcing $\int\psi\d\mu_{\delta_{0}\varphi+q\psi}=\partial/\partial q\mathcal{P}\left(\delta_{0}\varphi+q\psi\right)>0$
for some $-\epsilon<0$. By convexity of $q\mapsto\mathcal{P}(\delta_{0}\varphi+q\psi_{\alpha})$
the claim follows. Hence, there exists $s<\delta_{0}$ and $q<0$
with $\mathcal{P}\left(s\varphi+q\psi\right)\leq0$. Conversely, if
there exists $s<\delta_{0}$ and $q<0$ with $\mathcal{P}\left(s\varphi+q\psi\right)\leq0$,
then $\mathcal{P}\left(\delta_{0}\varphi+q\psi\right)<0$ and the
left derivative of $q\mapsto\mathcal{P}\left(\delta_{0}\varphi+q\psi\right)$
in $0$ must by convexity be positive. Since, using the variational
principle, $\alpha_{\max}=\int\psi\d\mu_{\delta_{0}\varphi}$ lies
in the subdifferential of $q\mapsto\mathcal{P}\left(\delta_{0}\varphi+q\psi\right)$
in $q=0$, we have $\alpha_{\max}>0$. This proves the first part.

If now $\alpha_{\max}\leq0$, then by \prettyref{lem:inf_sq_criterion}
$\delta_{0}=\inf\left\{ s>0:\inf_{q\leq0}\mathcal{P}\left(s\varphi+q\psi\right)\leq0\right\} $,
so a dimension gap can only occur if $\delta_{0}>\inf\left\{ s>0:\inf_{q>0}\mathcal{P}\left(s\varphi+q\psi\right)\leq0\right\} $.
If $\mathcal{P}\left(\delta_{0}\varphi+q\psi\right)=\infty$ for all
$q>0$ this is not possible. If on the other hand, $\mathcal{P}\left(\delta_{0}\varphi+q'\psi\right)<\infty$
some $q'>0$, then by the regularity assumption, the fact that $\mathcal{P}\left(\delta_{0}\varphi+q\psi\right)<\infty$
for all $q\leq q'$ and the convexity of $p$, we find that $p$ is
finite and strictly decreasing in the first coordinate on $\left\{ \left(s,q\right):q\in\left[0,q'\right],s\in[\delta_{0}-\left(q'-q\right)/\left(q'\right)\epsilon,\infty)\right\} $.
Since $\alpha_{\max}<0$ implies that $q\mapsto\mathcal{P}\left(\delta_{0}\varphi+q\psi\right)$
is strictly decreasing in a neighbourhood of $0$, there exists $s<\delta_{0}$
and $q>0$ with $\mathcal{P}\left(s\varphi+q\psi\right)\leq0$. If
finally $\alpha_{\max}=0$ and $\mathcal{P}\left(\delta_{0}\varphi+q'\psi\right)<\infty$
some $q'>0$, then $q\mapsto\mathcal{P}\left(\delta_{0}\varphi+q\psi\right)$
has a minimum in $0$ with value $0$. So there is no dimension gap.
\end{proof}
\begin{rem}
Finally, we would like to remark that by similar arguments as in the
proof of \prettyref{thm:DimGap_non-reflective} concerning the finiteness
and monotonicity of the pressure, $\inf_{q}\mathcal{P}\left(\delta_{0}\varphi+q\psi\right)<0$
implies the existence of $s<\delta_{0}$ and $q\in\R$, so that $\mathcal{P}\left(s\varphi+q\psi\right)\le0$.
Consequently, also $\inf_{q}\mathcal{P}\left(\delta_{0}\varphi+q\psi\right)<0$
is equivalent to the occurrence of a dimension gap.
\end{rem}

\subsection{Reflective case}

In this section, $g:\Sigma\rightarrow\R$ is supposed to be Hölder
continuous and $\psi,\psi_{1}:\Sigma\to\Z$ to be constant on 1-cylinders
with $\psi\geq-1$ and $M+\psi\geq\psi_{1}\geq\psi$.

The following lemma establishes a useful connection between the skew-product
system $\Psi_{\N}$ with reflective boundary and $\Psi_{\Z}$ the
corresponding system without reflective boundary; namely, the number
of visits to the reflective boundary corresponds to the maximum excursion
depth into the negative half-line of the system without reflective
boundary.

Obviously, $\FI_{\N}^{n}\left(x\right)-\FI_{\Z}^{n}\left(x\right)\ge0$
for all $x\in\Rep_{\N}$. In the following let $\left\Vert f\right\Vert $
denote the sup-norm of the real function $f$.
\begin{lem}
\label{lem:CombinatoriaKeyObservation} For all $x\in\Rep_{\N}\cap\left(0,1\right)$,
we have 
\[
\max_{0\leq k\leq n}-\left\lfloor \FI_{\Z}^{k}\left(x\right)\right\rfloor \leq\FI_{\N}^{n}\left(x\right)-\FI_{\Z}^{n}\left(x\right)\leq\max_{0\leq k\leq n}-\left\lfloor \FI_{\Z}^{k}\left(x\right)\right\rfloor +\left\Vert \psi_{1}-\psi\right\Vert 
\]
and
\[
\max_{0\leq k\leq n}-\left\lfloor \FI_{\Z}^{k}\left(x\right)\right\rfloor \leq\left\Vert \psi_{1}-\psi\right\Vert \cdot\card\left\{ 1\le j\le n:\FI_{\N}^{j}\left(x\right)\in\left[0,1\right]\right\} .
\]
\end{lem}

\begin{proof}
First we define the sequence of entrance times of $x\in\Rep_{\N}\cap\left(0,1\right)$
to $\left(0,1\right)$. Let $n_{0}\coloneqq0$ and $n_{j}\coloneqq\inf\left\{ n>n_{j-1}\colon\FI_{\N}^{n}\left(x\right)\in\left(0,1\right)\right\} $
for $j\geq1$; let $j_{0}\in\N$ be the smallest index such that $n_{j_{0}+1}=+\infty$
and $j_{0}\coloneqq+\infty$ if no such index exists.. Setting $a_{n}\coloneqq\FI_{\N}^{n}\left(x\right)-\FI_{\Z}^{n}\left(x\right)$,
we observe that $\left(a_{n}\right)$ is increasing with $a_{0}=0$
and, more precisely,
\begin{equation}
a_{n+1}=\begin{cases}
a_{n}+\left(\psi_{1}-\psi\right)\left(\FI_{0}^{n}\left(x-\left\lfloor x\right\rfloor \right)\right) & n=n_{j}\,\text{ for some }j\leq j_{0},\\
a_{n} & \text{else.}
\end{cases}\label{eq:an_recursion}
\end{equation}
Consequently, $a_{n_{j+1}}-a_{n_{j}}\leq\left\Vert \psi_{1}-\psi\right\Vert $.
On the other hand, setting $b_{n}\coloneqq\max_{0\leq k\leq n}-\left\lfloor \FI_{\Z}^{k}\left(x\right)\right\rfloor $,
$n\geq0$, we observe that $\left(b_{n}\right)$ is increasing and
non-negative. For all $j\in\N$, we find $b_{n_{j}}=-\left\lfloor \FI_{\Z}^{n_{j}}\left(x\right)\right\rfloor =a_{n_{j}}$.
To see this, observing that $a_{n_{0}}=a_{0}=0=b_{0}=b_{n_{0}}$ and,
by construction, $-\left\lfloor \FI_{\Z}^{n_{j}}\left(x\right)\right\rfloor =a_{n_{j}}$
and, for all $n_{j}<k\leq n_{j+1}$, we have $b_{k}\leq a_{k}$. In
particular, this implies $b_{n}\leq a_{n}$ for all $n\in\N$. Further,
for $n_{j}<k\leq n_{j+1},$$j\in\N_{<j_{0}}$, we deduce
\[
-\left\Vert \psi_{1}-\psi\right\Vert +a_{k}\leq-\left(a_{n_{j+1}}-a_{n_{j}}\right)+a_{k}\leq a_{n_{j}}=b_{n_{j}}\leq b_{k}\leq a_{k},
\]
which completes the proof of the first assertion. Using \prettyref{eq:an_recursion}
and the lower bound in the first assertion, we estimate
\[
\max_{0\leq k\leq n}-\FI_{\Z}^{k}\left(x\right)=b_{n}\le a_{n+1}=\sum_{j\in\N:n_{j}\le n}\left(\psi_{1}-\psi\right)\left(\FI_{0}^{n_{j}}\left(x-\left\lfloor x\right\rfloor \right)\right)\le\card\left\{ j\in\N:n_{j}\le n\right\} \left\Vert \psi_{1}-\psi\right\Vert ,
\]
which gives the second assertion.
\end{proof}
We now give a first basic inequality comparing the pressures with
resepct to $\H=0,\Z$ and $\N$.
\begin{prop}
\label{prop:Basic_Inequality} For $g:\Rep_{0}\rightarrow\R_{<0}$
H\"older continuous we have
\[
\mathcal{P}\left(g,\Z\right)\leq\mathcal{P}\left(g,\N\right)\leq\mathcal{P}\left(g\right).
\]
\end{prop}

\begin{proof}
The second inequality is obvious. To prove the first inequality, we
work with the coding $\pi=\pi_{\FI_{0}}:I^{\N}\rightarrow\overline{\Delta_{0}}$
for the base map $\FI$. There exists $\nu\in I^{2}$ such that the
basic interval $L=J_{\nu}$ satisfies $\overline{L}\subset(0,1)$.
Note that by transitivity of $\FI_{\N}$ for each $k\in\N$, there
exists a finite set $I_{0}=I_{0}(k)\subset I$ and $R_{k}\ge0$ such
that for each $0\le\ell\le\left\Vert \psi-\psi_{1}\right\Vert k$
there exists $s\le R_{k}$ and a word $\tau=\tau_{1}\cdots\tau_{s}\in I_{0}^{s}$
such that$\FI_{\N}^{s}(\ell+J_{\tau})\subset L$, Let $x\in\Per_{n}\left(\FI_{\Z}\right)\cap L$
with $\max_{0\leq m\leq n}\left|\FI_{\Z}^{m}x\right|\leq k$. Then,
by the skew-product property, $x\in\Per_{n}\left(\FI_{0}\right)$
and there exists a unique $\omega=\overline{\omega_{0}\cdots\omega_{n-1}}\in I^{\N}$
such that $x=\pi(\omega)$. By \prettyref{lem:CombinatoriaKeyObservation}
we have $\FI_{\N}^{n}(x)\le\left\Vert \psi-\psi_{1}\right\Vert k$.
Hence, for $\ell=\ell_{x}:=\left\lfloor \FI_{\N}^{n}(x)\right\rfloor $
there exists $s=s_{x}$ and $\tau\in I^{s}$ such that $\FI_{\N}^{n+s}(J_{\omega_{0}\cdots\omega_{n-1}\tau_{0}\cdots\tau_{s-1}})\subset(0,1)$.
We have $\overline{\omega\tau}\in\Per_{n+s}\left(\sigma\right)$ and
by \prettyref{lem:L_impies_in_Rep(f)} there exists $x'\in\Per_{n+s}\left(\FI_{0}\right)\cap L$
and the map $x\mapsto x'$ is one-to-one. Since $\FI_{\N}^{n+s}(J_{\omega_{0}\dots\omega_{n-1}\tau_{0}\dots\tau_{s-1}})\subset(0,1)$
we have $x'\in\Per_{n+s}\left(\FI_{\N}\right)\cap L$. Since $x,x'\in J_{\omega_{0}\cdots\omega_{n-1}}\subset L$
we have by the bounded distortion property

\begin{align*}
\sum_{{x\in\Per_{n}\left(\FI_{\Z}\right)\cap L\atop \max_{0\leq\ell\leq n}\left|\FI_{\Z}^{\ell}x\right|\leq k}}\e^{S_{n}g\left(x\right)} & \leq D_{g}\left(\sum_{i\in F_{k}}\max\e^{-g_{i}}\right)^{R_{k}}\sum_{\ell=0}^{R_{k}}\sum_{x'\in\Per_{n+\ell}\left(\FI_{\Z}\right)\cap L}\e^{S_{n+\ell}g\left(x'\right)}\\
 & \leq D_{g}\left(\sum_{i\in F_{k}}\max\e^{-g_{i}}\right)^{R_{k}}R_{k}\max_{0\leq\ell\leq R_{k}}\sum_{x'\in\Per_{n+\ell}\left(\FI_{\Z}\right)\cap L}\e^{S_{n+\ell}g\left(x'\right)}.
\end{align*}
By \prettyref{lem:exhaustio_principle} this implies $\mathcal{P}\left(g,\Z\right)\leq\mathcal{P}\left(g,\N\right)$.
\end{proof}
\begin{thm}[Variational Formula for $\N$-extensions]
\label{thm:Onesided_fibre-pressure_vs_base-pressure}For $g:\Rep_{0}\rightarrow\R_{<0}$
Hölder continuous we have
\[
\mathcal{P}\left(g,\N\right)=\inf_{q\leq0}\mathcal{P}(g+q\psi).
\]
\end{thm}

\begin{rem}
Note that $\inf_{q\in\R}\mathcal{P}(g+q\psi)=\mathcal{P}(g)$ implies
$\mathcal{P}\left(g,\Z\right)=\mathcal{P}\left(g\right)$.
\end{rem}

\begin{proof}
For $x\in L\cap\Rep_{\N}$, $L$ as above, and with $\FI_{\N}^{n}x=x$
we have $S_{n}\psi(x)\le0$ since $\psi\le\psi_{1}$. Hence, for $q\leq0$,
it follows that 
\begin{align*}
\sum_{x\in\Per_{n}\left(\FI_{\N}\right)\cap L}\e^{S_{n}g\left(x\right)} & \leq\sum_{x\in\Per_{n}\left(\FI_{0}\right)\cap L,S_{n}\psi\left(x\right)\leq0}\e^{S_{n}\left(g+q\psi\right)\left(x\right)}\\
 & \le\sum_{x\in\Per_{n}\left(\FI_{0}\right)\cap L}\e^{S_{n}\left(g+q\psi\right)\left(x\right)}
\end{align*}

This shows $\mathcal{P}\left(g,\N\right)\leq\inf_{q\leq0}\mathcal{P}\left(g+q\psi\right)$
and if $\inf_{q\leq0}\mathcal{P}(g+q\psi)=\inf_{q\in\R}\mathcal{P}(g+q\psi)$
then by\textbf{ }\prettyref{prop:Basic_Inequality} and \prettyref{thm:fibre_pressure_vs_base_pressure}
we have 
\[
\mathcal{P}\left(g,\Z\right)=\mathcal{P}\left(g,\N\right).
\]

Now, assume $\inf_{q>0}\mathcal{P}(g+q\psi)=\inf_{q\in\R}\mathcal{P}(g+q\psi)$.
Then $\mathcal{P}\left(g\right)=\infty$ implies $\mathcal{P}(g+q\psi)=\infty$
for all $q\in\R$ because $\psi\geq-1$. This gives $\mathcal{P}\left(g\right)=\mathcal{P}\left(g,\Z\right)=\infty$
and the equality follows from \prettyref{prop:Basic_Inequality}.

Hence, for the remaining case we can assume that $\mathcal{P}(g)<\infty$.
The obvious inequality $\mathcal{P}\left(g,\N\right)\le\mathcal{P}(g)$
has been established in \prettyref{prop:Basic_Inequality}. We now
prove the reverse inequality. Since $\inf_{q>0}\mathcal{P}(g+q\psi)=\inf_{q\in\R}\mathcal{P}(g+q\psi)$,
we find $\mathcal{P}(g+q\psi)<\infty$ on an open interval $(-\infty,\epsilon)$
for some $\epsilon>0$. By \cite{MR2003772} we conclude that $\int\psi\,d\mu_{g}=\partial/\partial q\left(\mathcal{P}(g+q\psi)\right)_{q=0}\leq0$.
Hence, we have $\liminf_{n}S_{n}\psi=-\infty$ $\mu_{g}$-a.\,e\@.
using the Law of Iterated Logarithm for the case $\int\psi\d\mu_{g}=0$
and the Law of Large Numbers for the case $\int\psi\d\mu_{g}<0$.
By the second assertion in \prettyref{lem:CombinatoriaKeyObservation},
for $\mu_{g}$-a.\,e\@. $x\in L$, we have $\FI_{\N}^{n}\left(x\right)\in\left(0,1\right)$
for infinitely many $n\in\N$. Hence, by the Borel--Cantelli Lemma
we have 
\[
\sum_{n\ge1}\mu_{g}\left\{ x\in\Rep_{\N}\cap L:\FI_{\N}^{n}\left(x\right)\in\left(0,1\right)\right\} =\infty.
\]
Recall that for each $x\in\Rep_{0}$ we define $\omega\left(x,n\right)\coloneqq\pi_{\FI_{0}}^{-1}\left(x\right)|_{n}\in I^{n}$.
Then, by the identification of periodic points stated in \prettyref{lem:L_impies_in_Rep(f)},
for every $n\in\N$, 
\[
\bigcup_{x\in\Per_{n}\left(\FI_{\N}\right)\cap L}J_{\omega\left(x,n\right)}\supset\left\{ x\in\Rep_{\N}\cap L:\FI_{\N}^{n}\left(x\right)\in\left(0,1\right)\right\} .
\]
With the Gibbs property of $\mu_{g}$ and \prettyref{eq:Skew_ergodic_sum}
we find 
\[
\infty=\sum_{n\ge1}\mu_{g}\left\{ x\in\Rep_{\N}\cap L:\FI_{\N}^{n}\left(x\right)\in\left(0,1\right)\right\} \ll\sum_{n\ge1}\sum_{x\in\Per_{n}\left(\FI_{\N}\right)\cap L}\e^{S_{n}g_{\N}\left(x\right)-n\mathcal{P}\left(g\right)}.
\]
From this it follows that $\mathcal{P}\left(g,\N\right)-\mathcal{P}\left(g\right)\ge0$
. This proves the remaining case.
\end{proof}
\begin{cor}
\label{cor:Onesided_vs_two_sided-1}For $g:\Rep_{0}\rightarrow\R_{<0}$
Hölder continuous we have
\begin{align*}
\mathcal{P}\left(g,\N\right) & =\begin{cases}
\mathcal{P}\left(g,\Z\right), & \text{if }\inf_{q\leq0}\mathcal{P}(g+q\psi)=\inf_{q\in\R}\mathcal{P}(g+q\psi),\\
\mathcal{P}(g), & \text{if }\inf_{q>0}\mathcal{P}(g+q\psi)=\inf_{q\in\R}\mathcal{P}(g+q\psi).
\end{cases}
\end{align*}
\end{cor}

\begin{proof}
If $\inf_{q\leq0}\mathcal{P}(g+q\psi)=\inf_{q\in\R}\mathcal{P}(g+q\psi)$
then the equality is an immediate consequence of \prettyref{thm:fibre_pressure_vs_base_pressure}
and \prettyref{thm:Onesided_fibre-pressure_vs_base-pressure}. Since
$\inf_{q>0}\mathcal{P}(g+q\psi)=\inf_{q\in\R}\mathcal{P}(g+q\psi)$
implies that $q\mapsto\mathcal{P}(g+q\psi)$ is by convexity decreasing
on $\R_{\leq0}$, it follows with \prettyref{thm:Onesided_fibre-pressure_vs_base-pressure}
that $\mathcal{P}\left(g,\N\right)=\inf_{q\leq0}\mathcal{P}(g+q\psi)=\mathcal{P}(g+0\psi)=\mathcal{P}(g)$.
\end{proof}
\begin{proof}[Proof of \prettyref{thm:HD_Transient}]
Since $\FI_{\N}^{n}\ge\FI_{\Z}^{n}$ we have $T^{+}\left(\FI_{\N}\right)\subset T^{+}\left(\FI_{\Z}\right)$.
By \prettyref{lem:CombinatoriaKeyObservation} we have $T^{+}\left(\FI_{\N}\right)+\left(-\N\right)\supset T^{+}\left(\FI_{\Z}\right)$,
thus, countable stability of the Hausdorff dimension proves the first
identity. The second identity holds trivially.

If $\alpha_{\max}>0$ then $\mu_{\delta_{0}\varphi}$-a.\,e\@. $x$
we have $\lim_{n}S_{n}\psi\left(x\right)=\infty$. Hence, $\delta_{0}\geq\dim_{H}T^{+}\left(\FI_{\Z}\right)\ge\dim_{H}\left(\mu_{\delta_{0}\varphi}\right)=\delta_{0}$.

Now suppose that $\alpha_{\max}\leq0$. Since $0\in\left(\underline{\psi},\overline{\psi}\right)$,
by \prettyref{thm:fibre_pressure_vs_base_pressure} there exists $q_{0}\geq0$
(cf\@. the proof of \cite[Proposition 4.3]{MR4373996}) such that
\[
\mathcal{P}\left(\delta_{\Z}\varphi+q_{0}\psi\right)=\min_{q}\mathcal{P}\left(\delta_{\Z}\varphi+q\psi\right)=\mathcal{P}(\delta_{\Z}\varphi,\Z)<\infty.
\]
Since $q_{0}\geq0$, arguing as in the proof of \cite[Theorem 1.3]{MR4373996},
for $\epsilon>0$, $k,N\in\N$, and since by \prettyref{lem:L_impies_in_Rep(f)}
every basic interval contained in $L$ contains a periodic point for
$\FI$, we can show $\dim_{H}\left(T^{+}\left(\FI_{\Z}\right)\cap L\right)\le\delta_{\Z}$.
Recall that for $x\in\Per_{n}\left(\FI_{0}\right)\cap L$there exists
a unique $\omega(x,n)\in I^{n}$ such that $x=\pi\left(\overline{\omega\left(x,n\right)}\right)$.
Hence, 
\begin{align*}
\sum_{n\geq N}\sum_{x\in\Per_{n}\left(\FI_{0}\right)\cap L,S_{n}\psi\left(x\right)>k}\left|J_{\omega\left(x,n\right)}\right|^{\delta_{\Z}+\epsilon} & \leq\sum_{n\geq N}\sum_{x\in\Per_{n}\left(\FI_{0}\right)}\exp\left(S_{n}\left(\left(\delta_{\Z}+\epsilon\right)\varphi+q_{0}\psi\right)\left(x\right)\right)<\infty
\end{align*}
which proves $\dim_{H}\left(T^{+}\left(\FI_{\Z}\right)\cap L\right)\le\delta_{\Z}$.
Finally, since $T^{+}\left(\FI_{\Z}\right)\subset\bigcup_{m\in\N}\FI_{\Z}^{m}\left(T^{+}\left(\FI_{\Z}\right)\cap L\right)$
and the branches of $\FI_{\Z}$ are Lipschitz, the countable stability
of Hausdorff dimension implies $\dim_{H}T^{+}\left(\FI_{\Z}\right)\le\delta_{\Z}$
follows.

For each $n\in\N$, we define $\delta_{\Z}^{n}$ for the restricted
map $\left(\FI_{0}\right)_{I\cap\{0,\dots,n\}}$ restricted to $\Delta_{I\cap\{0,\dots,n\}}$.
Using the exhaustion principle \prettyref{lem:exhaustio_principle},
\prettyref{thm:Onesided_fibre-pressure_vs_base-pressure} and \prettyref{lem:infimum_exchange}
shows that $\delta_{\Z}^{n}\rightarrow\delta_{\Z}$ as $n\rightarrow\infty$.

By \cite[Theorem 1.3]{MR4373996} we have for each $n\in\N$, $\dim_{H}T^{+}\left(\FI_{\Z}\right)\ge\delta_{\Z}^{n}$
completing the proof.
\end{proof}
\begin{proof}[Proof of \prettyref{thm:Big_Theorem}]
 Using \prettyref{thm:Onesided_fibre-pressure_vs_base-pressure}
we have $\mathcal{P}\left(s\varphi,\N\right)=\inf_{q\leq0}\mathcal{P}(s\varphi+q\psi)$.
Hence, by \prettyref{lem:infimum_exchange} we obtain
\[
\delta_{\N}=\inf\left\{ s>0:\inf_{q\leq0}\mathcal{P}(s\varphi+q\psi)\leq0\right\} =\inf_{q\le0}\inf\left\{ s\ge0:\mathcal{P}\left(s\varphi+q\psi\right)\le0\right\} 
\]
and by \prettyref{lem:inf_sq_criterion} the formula relating the
critical exponents follows.

We now turn to the proof of the trichotomy. If $\alpha_{\max}>0$
then $\delta_{\Z}<\delta_{0}$ by \prettyref{thm:DimGap_non-reflective}.
By the ergodic theorem we have $\mu_{\delta_{0}\varphi}\left(T^{+}\left(\FI_{\N}\right)\right)=1$.
If $\alpha_{\max}<0$, by the first part $\delta_{\N}=\delta_{0}$.
It is easy to see that in this case $\mu_{\delta_{0}\varphi}\left(R\left(\FI_{\N}\right)\right)=1$
and $\mu_{\delta_{0}\varphi}\left(T^{+}\left(\FI_{\N}\right)\right)=0$.
If $\alpha_{\max}=0$ then $\delta_{\Z}=\delta_{0}$ by \prettyref{thm:DimGap_non-reflective}.
Clearly, by the law of iterated logarithm, we have $\mu_{\delta_{0}\varphi}\left(R\left(\FI_{\N}\right)\right)=1$
and $\mu_{\delta_{0}\varphi}\left(T^{+}\left(\FI_{\N}\right)\right)=0$
.
\end{proof}

\subsection{Second-order phase-transition\label{subsec:Phase-transition}}

We close this section discussing the phenomenon of phase transition
which occurs naturally in the context of reflective systems.
\begin{proof}[Proof of \prettyref{thm:Phasetransition}]
  First note that by our assumptions, for all $s>s_{1}$ and $q<q_{0}$
we have $p(s,q):=\mathcal{P}\left(s\varphi+q\psi\right)<\infty$.
Hence, using convexity of $p$, we have that $p$ is finite in a neighbourhood
of $\left(s_{0},0\right)$ in $\R^{2}$ and consequently, in this
neighbourhood $\left(s,q\right)\mapsto\int\psi\d\mu_{s\varphi+q\psi}=\frac{\partial}{\partial q}p(s,q)$
is also finite and real-analytic satisfying $\frac{\partial}{\partial q}p(s,q)|_{\left(s_{0},0\right)}=0$
and $\frac{\partial^{2}}{\partial q^{2}}p(s,q)|_{\left(s_{0},0\right)}=\cov_{s_{0}\varphi}\left(\psi,\psi\right)$
\cite[Propositions 2.6.13 and  2.6.14]{MR2003772}. Our standing assumption
$\underline{\psi}<0<\overline{\psi}$ implies that $\cov_{s_{0}\varphi}\left(\psi,\psi\right)>0$.
Therefore, by the implicit function theorem there exists a neighbourhood
$U$ of $s_{0}$ and an real-analytic function $q$ on $U$ with $q\left(s_{0}\right)=0$,
\begin{align*}
\frac{\partial}{\partial q}p(s,q)|_{(s,q(s))} & =\int\psi\d\mu_{s\varphi+q(s)\psi}=0,\,\text{and }\\
q'(s) & =-\frac{\frac{\partial}{\partial s}\frac{\partial}{\partial q}p(s,q)|_{(s,q(s))}}{\frac{\partial}{\partial q}\frac{\partial}{\partial q}p(s,q)|_{(s,q(s))}}=-\frac{\cov_{s\varphi+q(s)\psi}\left(\varphi,\psi\right)}{\cov_{s\varphi+q(s)\psi}\left(\psi,\psi\right)}.
\end{align*}
By convexity and the fact that $\frac{\partial}{\partial q}p(s,q)|_{(s,q(s))}=0$,
we have $\mathcal{P}(s\varphi+q(s)\psi)=\min_{q\in\R}\mathcal{P}(s\varphi+q\psi)$,
for $s\in U$. By \prettyref{cor:Onesided_vs_two_sided-1} and \prettyref{thm:fibre_pressure_vs_base_pressure},
for all $s\in U$, we consequently have
\[
\mathcal{P}(s\varphi,\N)=\inf_{q\leq0}\mathcal{P}(s\varphi+q\psi)=\begin{cases}
\mathcal{P}(s\varphi+q(s)\psi), & q\left(s\right)\leq0,\\
\mathcal{P}(s\varphi), & q\left(s\right)>0.
\end{cases}
\]
The first and second derivative of $s\mapsto p(s,q(s))$ are then
\begin{align*}
\frac{\partial}{\partial s}p(s,q(s)) & =\frac{\partial}{\partial s}p(s,q)|_{(s,q(s))}+q'(s)\underbrace{\frac{\partial}{\partial q}p(s,q)|_{(s,q(s))}}_{=0}=\frac{\partial}{\partial s}p(s,q)|_{(s,q(s))}\qquad\text{and}\\
\frac{\partial^{2}}{\partial s^{2}}p(s,q(s)) & =\frac{\partial}{\partial s}\frac{\partial}{\partial s}p(s,q)|_{(s,q(s))}+q'(s)\frac{\partial}{\partial s}\frac{\partial}{\partial q}p(s,q)|_{(s,q(s))}\\
 & =\cov_{s\varphi+q(s)\psi}\left(\varphi,\varphi\right)-\frac{\left(\cov_{s\varphi+q(s)\psi}\left(\varphi,\psi\right)\right)^{2}}{\cov_{s\varphi+q(s)\psi}\left(\psi,\psi\right)}.
\end{align*}
First assume that $\cov_{s\varphi+q(s)\psi}\left(\varphi,\psi\right)>0$.
Then $q'(s_{0})<0$ and the one-sided first derivatives of $s\mapsto\mathcal{P}(s\varphi,\N)$
at $s_{0}$ from the left and right are given by 
\begin{align*}
\lim_{t\nearrow s_{0}}\frac{\partial}{\partial s}\mathcal{P}(s\varphi,\N)|_{s=t} & =\frac{\partial}{\partial s}p(s,0)|_{s=s_{0}}=\int\varphi\d\mu_{s_{0}\varphi},\\
\lim_{t\searrow s_{0}}\frac{\partial}{\partial s}\mathcal{P}(s\varphi,\N)|_{s=t} & =\frac{\partial}{\partial s}p(s,q(s))|_{s=s_{0}}=\int\varphi\d\mu_{s_{0}\varphi},
\end{align*}
and coincide, whereas the one-sided second derivatives turn out to
be
\begin{align*}
\lim_{t\nearrow s_{0}}\frac{\partial^{2}}{\partial s^{2}}\mathcal{P}(s\varphi,\N)|_{s=t} & =\frac{\partial^{2}}{\partial s^{2}}p(s,0)|_{s=s_{0}}=\cov_{s_{0}\varphi}\left(\varphi,\varphi\right),\\
\lim_{t\searrow s_{0}}\frac{\partial^{2}}{\partial s^{2}}\mathcal{P}(s\varphi,\N)|_{s=t} & =\frac{\partial^{2}}{\partial s^{2}}p(s,q(s))|_{s=s_{0}}=\cov_{s_{0}\varphi}\left(\varphi,\varphi\right)-\frac{\left(\cov_{s_{0}\varphi}\left(\varphi,\psi\right)\right)^{2}}{\cov_{s_{0}\varphi}\left(\psi,\psi\right)}.
\end{align*}
We conclude that $s\mapsto\mathcal{P}(s\varphi,\N)$ is continuously
differentiable at $s_{0}$, but not twice differentiable at $s_{0}$.
The case $\cov_{s\varphi+q(s)\psi}\left(\varphi,\psi\right)<0$ is
similar. If $\cov_{s\varphi+q(s)\psi}\left(\varphi,\psi\right)=0$,
then $s\mapsto\mathcal{P}(s\varphi,\N)$ is twice differentiable at
$s_{0}$ since 
\[
\lim_{t\nearrow s_{0}}\frac{\partial^{2}}{\partial s^{2}}\mathcal{P}(s\varphi,\N)|_{s=t}=\cov_{s_{0}\varphi}\left(\varphi,\varphi\right)=\lim_{t\searrow s_{0}}\frac{\partial^{2}}{\partial s^{2}}\mathcal{P}(s\varphi,\N)|_{s=t}.\qedhere
\]
\end{proof}
\printbibliography

\end{document}